\documentclass[11pt]{article}

\usepackage{vmargin}
\setmarginsrb{3cm}{1cm}{3cm}{1cm}{0.5cm}{0cm}{0cm}{0.5cm}

\usepackage{lmodern}

\usepackage{amsmath, amssymb, amsfonts, amsthm}
\usepackage{euscript}

\usepackage{enumerate}

\usepackage[dvips]{graphicx} 
\usepackage{caption}
\usepackage{subcaption} 

    \usepackage{authblk}

\DeclareMathAlphabet{\mathbbo}{U}{bbold}{m}{n}

\usepackage{todonotes}

\newcommand{\indic}{\mathbbo{1}}

\begin{document}

\numberwithin{equation}{section}

\newtheoremstyle{thmstyle}{3pt}{3pt}{\slshape}{}{\bf}{.}{.5em}{}      
\newtheoremstyle{defstyle}{3pt}{3pt}{\sffamily}{}{\bf}{.}{.5em}{}       

\theoremstyle{thmstyle}
\newtheorem{thm}{Theorem}[section]
\newtheorem{ppsition}[thm]{Proposition}

\newtheorem{lemma}[thm]{Lemma}
\newtheorem{cor}[thm]{Corollary}

\theoremstyle{defstyle}
\newtheorem{definition}[thm]{Definition}
\newtheorem{example}[thm]{Example}
\newtheorem*{example*}{Example}
\newtheorem{remark}[thm]{Remark}
\newtheorem{notation}[thm]{Notation}
\newtheorem{question}[thm]{Question}

\newcommand{\wt}[1]{\widetilde{#1}}

\newcommand{\var}{\mathop{\rm Var}}

\renewcommand{\AA}{{\EuScript A}}
\newcommand{\BB}{{\EuScript B}}
\newcommand{\CC}{{\EuScript C}}
\newcommand{\DD}{{\EuScript D}}
\newcommand{\EEE}{\EuScript{E}}
\newcommand{\FF}{\EuScript{F}}
\newcommand{\GG}{\EuScript{G}}
\newcommand{\HH}{\EuScript{H}}
\newcommand{\II}{\EuScript{I}} 
\newcommand{\JJ}{\EuScript{J}}
\newcommand{\LL}{\EuScript{L}}
\newcommand{\MM}{\EuScript{M}}
\newcommand{\N}{\mathbb{N}}
\renewcommand{\P}{\mathcal{P}}
\newcommand{\PP}{\mathbb{P}}
\newcommand{\RR}{\mathbb{R}}
\renewcommand{\SS}{\EuScript{S}}
\newcommand{\UU}{\EuScript{U}}
\newcommand{\VV}{{\EuScript V}}
\newcommand{\Z}{\mathbb{Z}}
\newcommand{\EE}{\mathbb{E}}
\newcommand{\EEtil}{\widetilde{\EE}}
\newcommand{\PPbar}{\overline{\mathbb{P}}}
\newcommand{\sig}{\sigma}
\newcommand{\eps}{\varepsilon}

\newcommand{\Eb}{\boldsymbol{E}}
\newcommand{\Vb}{\boldsymbol{V}}

\newcommand{\OmAP}{(\Omega, {\EuScript A}, \mathbb{P})}
\newcommand{\OmAPprime}{(\Omega', {\EuScript A}', \mathbb{P}')}

\newcommand{\EEEE}{\EEE={(\EEE_n)}_{n \leq 0}}
\newcommand{\FFFF}{\FF={(\FF_n)}_{n \leq 0}}
\newcommand{\GGGG}{\GG={(\GG_n)}_{n \leq 0}}
\newcommand{\UUUU}{\UU={(\UU_n)}_{n \leq 0}}
\newcommand{\VVVV}{\VV={(\VV_n)}_{n \leq 0}}

\newcommand{\longto}{\longrightarrow}

\newcommand{\given}{\, | \,}

\let\leq=\leqslant
\let\geq=\geqslant

\def\dd{\mathrm{dd}}%

\newcommand{\myminipage}[3]{
\begin{minipage}[t][#2][c]{#1}
#3
\end{minipage}
}

\newcommand{\leqst}{\overset{\text{st}}{\leq}}
	\newcommand{\tend}[3][]{\xrightarrow[#2\to#3]{#1}}
	\def\egdef{:=}

\renewcommand{\Pr}{{\PP}}
\newcommand{\Bin}{\text{Bin}}

\newcommand\numberthis{\addtocounter{equation}{1}\tag{\theequation}}

\title{Standardness of monotonic Markov filtrations}
\date{}
\author[1]{\'Elise Janvresse\thanks{elise.janvresse@univ-rouen.fr}} 
\author[2]{Stéphane Laurent\thanks{laurent\_step@yahoo.fr}}
\author[1]{Thierry de la Rue\thanks{thierry.de-la-rue@univ-rouen.fr}}

\affil[1]{Laboratoire de Math\'ematiques Rapha\"el Salem,
	Normandie~Université, Universit\'e~de~Rouen, CNRS}
\affil[2]{Independent researcher}
        
\maketitle

\begin{abstract}
We derive a practical standardness criterion for the filtration generated by 
a monotonic Markov process. 
This criterion is applied to show standardness  of some adic filtrations.  
\end{abstract}

{\bf Keywords: }Standardness of filtrations; monotonic Markov processes; Adic filtrations; Vershik's intrinsic metrics.

{\bf MSC classification: } 60G05, 60J05, 60B99, 05C63.

\tableofcontents

\section{Introduction}

The theory of filtrations in discrete negative time was originally 
developed by Vershik in the 70's. 
It mainly deals with the identification of standard filtrations. 
Standardness is an invariant property of filtrations $\FFFF$ in discrete negative time, whose definition is recalled below (Definition~\ref{def:standard}). 
It only concerns the case when the $\sigma$-field $\FF_0$ is essentially separable, 
and in this situation one can always find a Markov process\footnote{By \emph{Markov process} 
we mean any stochastic process ${(X_n)}_{n \leq 0}$ satisfying the Markov property, 
but no stationarity and no homogeneity in time are required.}
${(X_n)}_{n \leq 0}$ that generates the filtration $\FF$ by taking 
for $X_n$ any random variable generating the $\sigma$-field $\FF_n$ for every $n \leq 0$. 

 In Section~\ref{sec:markovstandard}, we provide two 
 standardness criteria for a filtration 
 given as generated by a  Markov process. 
 The first one, Lemma~\ref{lemma:PWgenerating}, is a somewhat elementary criterion involving a 
 construction we call the \emph{Propp-Wilson coupling} (Section~\ref{sec:PWcoupling}). 
The second one, Lemma~\ref{lemma:VershikMarkov}, is borrowed from~\cite{LauEntropy}. It  
 is a particular form of Vershik's standardness criterion 
which is known to be equivalent to standardness (see~\cite{ES}). 

The main result of this paper is stated and proved in Section~\ref{sec:monotonicmarkov} (Theorem~\ref{thm:monotonic}): It provides a very convenient standardness criterion for 
 filtrations which are given as generated by a \emph{monotonic} 
 Markov process ${(X_n)}_{n \leq 0}$ (see Definition~\ref{def:monotonic}). 
 It is generalized in Section~\ref{sec:multimarkov} (Theorem \ref{thm:monotonicmulti}) 
 to multidimensional Markov processes.

There is a revival interest in standardness due to the recent works 
of Vershik~\cite{Ver13,VerIntrinsic,VerGraded} which connect the theory of filtrations 
to the problem of identifying ergodic central measures on Bratteli graphs, 
which is itself closely connected to other problems of mathematics. 
As we explain in Section~\ref{sec:adic}, an ergodic central measure 
on (the path space of) a Bratteli graph generates a filtration 
we call an adic filtration, and the recent discoveries by Vershik 
mainly deal with standardness of adic filtrations. 
Using our standardness criterion for the filtration of a monotonic 
Markov process, 
we show standardness for some adic filtrations 
arising from the Pascal graph and the Euler graph 
in the subsequent sections~\ref{sec:Pascal}, 
\ref{sec:Euler} and~\ref{sec:mpascal}. 
As a by-product, our results also provide a new proof of ergodicity 
of some adic transformations on these graphs. 
We also discuss the case of non-central measures.

%

\subsection{Standardness}

A filtration $\FFFF$ is said to be \emph{immersed} in a filtration $\GGGG$ if $\FF \subset \GG$ and for each $n< 0$, the \hbox{$\sigma$-field} $\FF_{n+1}$ is conditionally independent of $\GG_n$ given $\FF_n$. 
When $\FF$ is the filtration generated by a Markov process  ${(X_n)}_{n \leq 0}$, then 
saying that $\FF$ is immersed in some filtration $\GG$ tantamounts to say that $\FF \subset \GG$ 
and that ${(X_n)}_{n \leq 0}$ has the Markov property with respect to the bigger filtration 
$\GG$, that is, 
$${\cal L}(X_{n+1} \given \GG_{n}) = {\cal L}(X_{n+1} \given \FF_{n}) ={\cal L}(X_{n+1} \given X_{n})$$ 
for every $n<0$. 

A filtration is said to be of \emph{product type} if it is generated by a sequence of 
independent random variables.

\begin{definition}\label{def:standard}
A filtration $\FF$ is said to be \emph{standard} when it is immersed in a filtration 
of product type, possibly up to isomorphism (in which case we 
say that  $\FF$ is \emph{immersible} in a  filtration of product type).
\end{definition}

When ${(X_n)}_{n \leq 0}$ is any stochastic process generating the filtration $\FF$, 
then a filtration isomorphic to $\FF$ is a filtration generated by a 
\emph{copy} of ${(X_n)}_{n \leq 0}$, that is to say a stochastic process 
${(X'_n)}_{n \leq 0}$ defined on any probability space and having the same law 
as  ${(X_n)}_{n \leq 0}$.  

By Kolmogorov's $0$-$1$ law, a necessary condition for standardness is that 
the filtration $\FF$ be \emph{Kolmogorovian}, that is to say that 
the tail $\sigma$-algebra $\FF_{-\infty}$ be degenerate\footnote{The introduction of 
the word \emph{Kolmogorovian} 
firstly occured in \cite{LauXLIII} and \cite{LauTeoriya} and was motivated by the 
so-called Kolmogorov's $0$-$1$ law in the case of a product type filtration. 
By the correspondance between $(-\N)$-indexed filtrations and 
$\N$-indexed decreasing sequences of measurables partitions, one could also say \emph{ergodic}, 
because 
this property is equivalent to ergodicity of the equivalence relation defined by 
the tail partition.}.


\subsection{Generating parameterization criterion}

We prove in this section that a filtration having a generating 
parameterization is standard, after introducing the required definitions.
Constructing a generating parameterization is a frequent way 
to establish standardness in practice.

\begin{definition}\label{def:superinnovation}
Let $\FFFF$ be a filtration. 
A \emph{parameterization} of $\FF$   is a sequence of (independent) random
variables  ${\boldsymbol U}={(U_n)}_{n \leq 0}$ 
such that for each $n \leq 0$, the random variable $U_n$ 
 is independent of ${\FF_{n-1} \vee \sig(U_m;  m \leq n-1)}$,
 and satisfies $\FF_n \subset \FF_{n-1} \vee \sigma(U_n)$. 
 We say that the parameterization ${\boldsymbol U}$ is \emph{generating} if 
$\FF \subset \UU$, where $\UU$ is the filtration generated by ${\boldsymbol U}$. 
\end{definition}


It is shown in~\cite{LauXLIII} that, up to isomorphism, every 
 filtration $\FF$ having an essentially separable 
$\sigma$-field $\FF_0$ has 
a parameterization ${(U_n)}_{n \leq 0}$ where each $U_n$ has a uniform distribution 
on $[0,1]$. 

The following lemma is shown in~\cite{LauTeoriya}. 
It is the key point to show that a filtration having a generating 
parameterization is standard (Lemma~\ref{lemma:generatingparam}). 

\begin{lemma}\label{lemma:extension_superinnovation}
Let $\FF$ be a filtration having a parameterization 
 ${\boldsymbol U}={(U_n)}_{n \leq 0}$, and  
let $\UUUU$ be the filtration generated by ${\boldsymbol U}$. 
Then $\FF$ and $\UU$ are both immersed in  the filtration $\FF \vee \UU$.
\end{lemma}

\begin{definition}\label{def:extension_superinnovation}
The filtration  $\FF \vee \UU$ in the above lemma is called the
 \emph{extension of $\FF$ with the parameterization} ${\boldsymbol U}$, and 
is also said to be a \emph{parametric extension} of $\FF$.
\end{definition}

\begin{lemma}\label{lemma:generatingparam}
If ${\boldsymbol U}$ is a generating parameterization of the filtration $\FF$, then $\FF$ as well as 
$\FF\vee\UU$ are standard. 
\end{lemma}

\begin{proof}
Obviously  $\FF \vee \UU$ is standard because 
 $\UU$ is standard (even of product type), and $\FF \vee \UU=\UU$ under the generating assumption.  
Then the filtration $\FF$ is standard as well, because by Lemma~\ref{lemma:extension_superinnovation} it is immersed in the filtration $\UU$.
\end{proof}

Whether any standard filtration admits a generating parameterization is an open 
question of the theory of filtrations. 

\section{Standardness for the filtration of a Markov process}\label{sec:markovstandard}

From now on, we consider a Markov process ${(X_n)}_{n \leq 0}$ where, for each $n$, $X_n$ takes its values in a standard Borel space $A_n$, and whose transition probabilities are given by the sequence of kernels ${(P_n)}_{n\le 0}$: For each $n\le0$ 
 and each measurable subset $E\subset A_n$, 
$$ \PP(X_n\in E \given X_{n-1}) = P_n(X_{n-1},E) \quad\text{\it a.s.} $$

We denote by $\FF$ the filtration generated by ${(X_n)}_{n \leq 0}$. 
In this section, we provide two practical criteria to establish standardness of 
$\FF$: 
the Propp-Wilson coupling in  Section~\ref{sec:PWcoupling} 
(Lemma~\ref{lemma:PWgenerating}) and 
a simplified form of Vershik's standardness criterion in Section~\ref{sec:vershik} 
(Lemma~\ref{lemma:VershikMarkov}, borrowed from~\cite{LauEntropy}). 
Recall that any filtration having an essentially separable final $\sigma$-field $\FF_0$ can always be generated by a Markov process ${(X_n)}_{n \leq 0}$. 
But practicality of the standardness criteria we present in this section 
lies on the choice of the generating Markov process. 

The Propp-Wilson coupling is a practical criterion 
to construct a generating parameterization of $\FF$. 
It will be used to prove our standardness criterion for monotonic Markov 
processes (Theorem~\ref{thm:monotonic}) which is the main 
result of this article. 
The simplified form of Vershik's standardness criterion we provide 
in Lemma~\ref{lemma:VershikMarkov} will not be used to prove  Theorem~\ref{thm:monotonic}, 
but the iterated Kantorovich pseudometrics $\rho_n$ introduced to state this 
criterion will play an important role in the proof of Theorem~\ref{thm:monotonic}, 
and they will also appear in Section~\ref{sec:adic} as the \emph{intrinsic metrics} 
in the particular context of adic filtrations. 
Lemma~\ref{lemma:VershikMarkov} itself will only be used in section 
\ref{sec:mpascal}. 

 The general statement of Vershik's standardness criterion concerns
 an arbitrary filtration $\FF$ 
 and it is known to be equivalent to standardness as long as the final $\sigma$-field 
 $\FF_0$ is essentially separable. 
Its statement is simplified in Lemma~\ref{lemma:VershikMarkov}, mainly because 
it is specifically stated for the case when $\FF$ is the filtration of 
the Markov process ${(X_n)}_{n \leq 0}$, together with an identifiability assumption  
on the Markov kernels $P_n$.

\subsection{Markov updating functions and the Propp-Wilson coupling}\label{sec:PWcoupling} 

For the filtration $\FF$ generated by the Markov process ${(X_n)}_{n \leq 0}$, it is possible to have, up to isomorphism,
 a parameterization ${(U_n)}_{n \leq 0}$  of $\FF$ with the additional property
 $$\sig(X_{n}) \subset \sig(X_{n-1},U_{n}) \quad\text{for each $n\le0$}.$$  
This fact is shown in~\cite{LauTeoriya} but we will consider it from another point of view here. 
The above inclusion  means that $X_n = f_n(X_{n-1},U_n)$ for some measurable function $f_n$.
Such a function is appropriate when it is an \emph{updating function} of the 
Markov kernel $P_n$, that is to say a measurable function 
$f_n\colon (x,u)\mapsto f_n(x,u)\in A_n$ such that  
$f_n(x,\cdot)$ sends the distribution law of $U_n$ to $P_n(x,\cdot)$
for each $x \in A_{n-1}$. 

Such updating functions, associated to random variables $U_n$ which are uniformly distributed in $[0,1]$, always exist. Indeed, there is no loss of generality to assume that 
each $X_n$ takes its values in $\RR$. Then, the most common choice of 
$f_n$ is the \emph{quantile updating function}, defined as the inverse of the right-continuous cumulative 
distribution function of the conditional law $\LL(X_n \given X_{n-1}=x)=P_n(x,\cdot)$: 
\begin{equation}
  \label{eq:quantile}\mbox{For }0<u<1, \quad 
f_n(x, u) = \inf\left\{t \in\RR: \PP\bigl(X_n\le  t\ |\ X_{n-1}=x\bigr)\ge u\right\}.
\end{equation}

Once the updating functions $f_n$ are given, it is not difficult to get, up to isomorphism, 
a parameterization ${(U_n)}_{n \leq 0}$ for which $X_n = f_n(X_{n-1},U_n)$, using the Kolmogorov extension theorem. We then say that ${(X_n)}_{n \leq 0}$  \emph{is parameterized by}  ${(f_n, U_n)}_{n \leq 0}$ 
 and that  ${(f_n, U_n)}_{n \leq 0}$ is a \emph{parametric representation} 
 of ${(X_n)}_{n \leq 0}$.

\medskip   

Given a parametric representation ${(f_n, U_n)}_{n\le0}$ of 
${(X_n)}_{n \leq 0}$,  
the \emph{Propp-Wilson coupling} is a practical tool to check whether ${(U_n)}_{n \leq 0}$ 
 is a generating parameterization 
of the filtration $\FF$ generated by  ${(X_n)}_{n \leq 0}$. 
 Given $n_0 \leq -1$ and a point $x_{n_0}$ in $A_{n_0}$, 
there is a natural way to construct, on the same probability space,
 a Markov process ${\bigl(Y_n(n_0,x_{n_0})\bigr)}_{n_0 \leq n \leq 0}$ with initial condition 
 $Y_{n_0}(n_0,x_{n_0})=x_{n_0}$ and having the same transition kernels as 
 ${(X_n)}_{n_0 \leq n \leq 0}$:  
 It suffices to set the initial condtion $Y_{n_0}(n_0,x_{n_0})=x_{n_0}$ and to 
 use the inductive relation
 $$ \forall n_0\le n < 0,\quad Y_{n+1}(n_0,x_{n_0})\egdef f_{n+1}\Bigl(Y_n(n_0,x_{n_0}),U_{n+1}\Bigr). $$

 We call this construction the \emph{Propp-Wilson coupling} because it is a well-known 
 construction used in Propp and Wilson's coupling-from-the-past algorithm~\cite{PW}. 
The word ``coupling'' refers to the fact that the random variables $Y_n$ are constructed on the same probability 
space as the Markov process  ${(X_n)}_{n \leq 0}$. 
The following lemma shows how to use the Propp-Wilson coupling to prove 
the generating property of ${(U_n)}_{n \leq 0}$.

\begin{lemma}\label{lemma:PWgenerating}
Assume that, for every $n\le 0$, the state space $A_n$ of $X_n$ is Polish under some distance $d_n$ and 
that $\EE\Bigl[ d_n(X_n, Y_n(m,x_m)\bigr)\Bigr] \to 0$ as $m \to -\infty$ for some 
sequence ${(x_m)}$ (possibly depending on $n$) such that $x_m \in A_m$.  
Then 
${(U_n)}_{n \leq 0}$ is a generating parameterization 
of the filtration $\FF$ generated by ${(X_n)}_{n \leq 0}$. In particular, $\FF$ is standard.
\end{lemma}

\begin{proof} 
The assumption implies that every $X_n$ is measurable with respect to 
$\sigma(\ldots, U_{n-1},U_n)$ because 
 $Y_n(m,x_m)$ is  $\sigma(U_{m+1}, \ldots, U_n)$-measurable. 
Then it is easy to check that ${(U_n)}_{n\le 0}$ is a generating parameterization of 
$\FF$.
\end{proof}

\subsection{Iterated Kantorovich pseudometrics and Vershik's criterion}\label{sec:vershik}

Vershik's standardness criterion will only be necessary to prove the second multidimensional 
version of Theorem~\ref{thm:monotonic} (Theorem~\ref{thm:fullymonotonicmulti}). 
However the iterated Kantorovich pseudometrics lying at the heart of 
Vershik's standardness will be used in the proof of Theorem~\ref{thm:monotonic}. 

A \emph{coupling} of two probability measures $\mu$ and  $\nu$ is a pair $(X_\mu,X_\nu)$ of two random variables defined on the same probability space with respective distribution $\mu$ and $\nu$. When $\mu$ and $\nu$ are defined on the same separable 
metric space $(E,\rho)$, the \emph{Kantorovich distance} between $\mu$ and $\nu$ is defined by
\begin{equation}\label{eq:def_Kantorovich}
  \rho'(\mu,\nu) := \inf \EE[\rho(X_\mu,X_\nu)],
\end{equation}
where the infimum is taken over all couplings $(X_\mu,X_\nu)$ of $\mu$ and  $\nu$. 

If $(E,\rho)$ is compact, the weak topology on the set of probability measures on $E$ is itself compact and metrized by the Kantorovich metric $\rho'$. 
If $\rho$ is only a pseudometric on $E$, one can define $\rho'$ in the same way, but we only get a pseudometric on the set of probability measures.

\bigskip

The iterated Kantorovich pseudometrics $\rho_n$ defined below arise from 
the translations of Vershik's ideas~\cite{Ver95} into the context of our Markov process 
${(X_n)}_{n \leq 0}$. 
Let $n_0 \leq 0$ be an integer and 
assume that we are given a compact pseudometric 
$\rho_{n_0}$ on the state space $A_{n_0}$ of $X_{n_0}$. 
Then for every $n \leq n_0$  
we recursively define a compact pseudometric $\rho_n$ on the state space $A_n$
of $X_n$ by setting  
$$ \rho_n(x_n, x'_n) \egdef
(\rho_{n+1})'\bigl(P_n(x_n,\cdot),P_n(x'_n,\cdot)\bigr) $$
where $(\rho_{n+1})'$ is the Kantorovich pseudometric derived 
from $\rho_{n+1}$  as explained above.

\begin{definition}\label{def:vershikianmarkov}
With the above notations, 
we say that the random variable $X_{n_0}$ satisfies the \emph{$V'$ property} 
if $\EE\left[\rho_n(X'_n,X''_n)\right] \to 0$  
where $X'_n$ and $X''_n$ are two 
independent copies of $X_n$.  
\end{definition} 

Note that the $V'$ property of $X_{n_0}$ is not only a property 
of the random variable $X_{n_0}$ alone, since its statement relies on 
the Markov process ${(X_n)}_{n \leq 0}$. 
Actually the $V'$ property of 
$X_{n_0}$ is a rephrasement of the \emph{Vershik property} (not stated 
in the present paper)   
of $X_{n_0}$ with respect to the filtration $\FF$ 
generated by ${(X_n)}_{n \leq 0}$,  in the present 
context when ${(X_n)}_{n \leq 0}$ is a Markov process. 
The equivalence between these two properties is shown in 
\cite{LauEntropy}, but in the present paper we do not introduce 
the general Vershik property. 
The definition also relies on the choice of the initial compact pseudometric $\rho_{n_0}$, 
but it is shown in \cite{LauTeoriya} and \cite{LauEntropy} that 
the Vershik property of $X_{n_0}$ (with respect to $\FF$) 
and actually is a property about the $\sigma$-field $\sigma(X_{n_0})$ 
generated by $X_{n_0}$ and thus it does not depend on $\rho_{n_0}$.  
Admitting this equivalence between the $V'$ property and the 
Vershik property, and using proposition 6.2 in \cite{LauTeoriya}, 
we get the following proposition.

\begin{ppsition}\label{ppsition:vershikmarkov}
The filtration generated by the Markov process ${(X_n)}_{n \leq 0}$ is 
standard if and only if $X_n$ satisfies the $V'$ property  for every 
$n \leq 0$. 
\end{ppsition}

As shown in \cite{LauEntropy}, there is a considerable simplification 
 of 
Proposition \ref{ppsition:vershikmarkov} under the identifiability 
condition defined below. This is rephrased in Lemma~\ref{lemma:VershikMarkov}.

\begin{definition}\label{def:identifiable}
A Markov kernel $P$ is \emph{identifiable} when $x \mapsto P(x,\cdot)$ is one-to-one. 
A Markov process ${(X_n)}_{n \leq 0}$ is \emph{identifiable} if for every $n \leq 0$ 
its transition distributions $\LL(X_n \given X_{n-1}=x)$ are given 
by an identifiable Markov kernel $P_n$. 
\end{definition}

If $\rho_{n_0}$ is a metric and the Markov process is 
identifiable, then it is easy to prove by induction that $\rho_n$ 
is a metric for all $n\le n_0$, using the fact that $(\rho_{n+1})'$ is itself a metric.
Lemma~\ref{lemma:VershikMarkov} below, borrowed from \cite{LauEntropy}, 
provides a friendly 
statement of Vershik's standardness criterion 
for the filtration of an identifiable Markov process. 

\begin{lemma}\label{lemma:VershikMarkov}
Let ${(X_n)}_{n \leq 0}$ be an identifiable Markov process with 
$X_0$ taking its values in a compact metric space $(A_0,\rho_0)$.  
Then the filtration generated by ${(X_n)}_{n \leq 0}$ is standard if 
and only if $X_0$ satisfies the $V'$ property. 
\end{lemma}

%

\section{Monotonic Markov processes}\label{sec:monotonicmarkov}

Theorem~\ref{thm:monotonic} in Section~\ref{Section:standardness criterion} provides a simple standardness criterion 
for the filtration of a monotonic Markov process. 
After defining this kind of Markov processes, we introduce a series of tools before proving the theorem. 
An example is provided in this section (the Poissonian Markov chain), and 
examples of adic filtrations will be provided in Section~\ref{sec:adic}.

\subsection{Monotonic Markov processes and their representation}\label{sec:monotonicrepresentation}

\begin{definition}\label{def:ordered}
Let $\mu$ and $\nu$ be two probability measures  on the same ordered set, we say that the coupling $(X_\mu,X_\nu)$ of $\mu$ and $\nu$ is an \emph{ordered coupling} if 
$\Pr(X_\mu \le X_\nu)=1$ or $\Pr(X_\nu \le X_\mu)=1$. 
\end{definition} 


\begin{definition}
Let $\mu$ and $\nu$ be two probability measures on an ordered set. We say that $\mu$ is \emph{stochastically dominated} by $\nu$, and note $\mu\leqst\nu$, if there exists an ordered coupling $(X_\mu,X_\nu)$ such that $X_\mu\le X_\nu$ a.s.
\end{definition}

\begin{definition}\label{def:monotonic}
\begin{itemize}
\item When $A$ and $B$ are  ordered, 
a Markov kernel $P$ from  $A$ to  $B$ 
is \emph{increasing} if $x \leq x' \implies P(x,\cdot) \leqst P(x',\cdot)$. 

\item Let ${(X_n)}_{n \leq 0}$  be a Markov process such that each $X_n$ takes 
its values in an ordered set. 
We say that ${(X_n)}_{n \leq 0}$ is \emph{monotonic} if the Markov kernel 
$P_n(x,\cdot) \egdef \LL(X_{n} \given  X_{n-1}=x)$ is increasing for each $n$.
\end{itemize}
\end{definition} 

\begin{example}[{\bf Poissonian Markov chain}]\label{exple:poisson}
Given a decreasing sequence ${(\lambda_n)}_{n \leq 0}$ of positive real numbers, 
define the law of a Markov process ${(X_n)}_{n \leq 0}$ by:
\begin{itemize}
	\item \emph{(Instantaneous laws)} each $X_n$ has the Poisson distribution with mean $\lambda_n$;

	\item  \emph{(Markovian transition)} given $X_n=k$, the random variable $X_{n+1}$ has the 
binomial distribution on $\{0, \ldots, k\}$ with success probability parameter $\lambda_{n+1}/\lambda_n$. 
\end{itemize}
It is easy to check that the binomial distribution $\LL(X_{n+1} \given X_n=k)$ is stochastically 
increasing in $k$, hence  ${(X_n)}_{n \leq 0}$ is a monotonic Markov process. 
Note that it is identifiable (Definition~\ref{def:identifiable}). 
\end{example}

The notion of updating function for a Markov kernel has been introduced in 
Section~\ref{sec:PWcoupling}. 
Below we define the notion of increasing updating function, in 
the context of monotonic Markov kernels.

\begin{definition}
\begin{itemize}
\item Let $P$ be an (increasing) Markov kernel from  $A$ to  $B$ and 
$f$ be an updating function of $P$. We say that $f$ is 
an \emph{increasing updating function} if $f(x,u)\leq f(x',u)$ for almost all $u$ and for every 
$x,x' \in A$ satisfying $x \leq x'$.

\item We say that a parameterization ${(f_n, U_n)}_{n \leq 0}$ 
(defined in Section~\ref{sec:PWcoupling}) of a (monotonic) Markov process is an \emph{increasing 
representation} if every $f_n$ is  an increasing updating function, that is, 
the equality $f_n(x,U_n) \leq f_n(x',U_n)$ 
almost surely holds whenever $x \leq x'$. 
\end{itemize}
\end{definition} 

For a real-valued monotonic Markov process, it is easy to check that the quantile 
updating functions defined by~\eqref{eq:quantile} provide an increasing representation. 

%

\subsection{Standardness criterion for monotonic Markov processes}
\label{Section:standardness criterion}
The achievement of the present section is the following Theorem~which 
provides a practical criterion to check standardness of a filtration 
generated by a monotonic Markov process. 

\begin{thm}\label{thm:monotonic}
Let  ${(X_n)}_{n \leq 0}$  be an $\RR$-valued monotonic Markov process, 
and $\cal F$ the filtration it generates.
\begin{enumerate}[{1)}]
\item   
 The following conditions are equivalent. 
	\begin{enumerate}[{(a)}]
	\item $\FF$ is standard.
	\item $\FF$ admits a generating parameterization.
	\item Every increasing representation provides a generating parameterization. 
	\item For every $n \leq 0$, the conditional law ${\cal L}(X_{n} \given \FF_{-\infty})$ is almost surely equal to ${\cal L}(X_{n})$.
	\item $\FF$ is Kolmogorovian.
	\end{enumerate}

\item Assuming  that the Markov process is identifiable (Definition~\ref{def:identifiable}), then the five conditions above are 
equivalent to the almost-sure equality between the conditional law ${\cal L}(X_{0} \given \FF_{-\infty})$ and ${\cal L}(X_{0})$.
\end{enumerate}
\end{thm}

Before giving the proof of the theorem, we isolate the main tools that we will use. 

\subsubsection{Tool 1: Convergence of $\LL(X\given\FF_n)$}

Lemma \ref{lemma:degenerate}  is somehow a rephrasement of 
L\'evy's reversed martingale convergence theorem. 
It says in particular that condition (d) of Theorem~\ref{thm:monotonic} is the same 
as the convergence $\LL(X_n \given \FF_m) \tend{m}{-\infty} \LL(X_n)$. 
We state a preliminary lemma which will also be used in Section \ref{sec:monotonicmulti}. 

Given, on some probability space, a $\sigma$-field $\cal B$ and 
a random variable $X$ taking its values in a Polish space $A$, the conditional 
law $\LL(X \given {\cal B})$ is a random variable when the narrow topology 
is considered on the space of probability measures on $A$, and this topology 
coincides with the topology of weak convergence when $A$ is compact 
(see \cite{Crauel}). 

\begin{lemma}\label{lemma:convergence_conditional_law}
Let $A$ be a compact metric space and 
${(\Gamma_k)}_{k \geq 0}$ a sequence of random variables taking values in the space of probability 
measures on $A$ equipped   
with the topology of weak convergence. 
Then the sequence
 ${(\Gamma_k)}_{k \geq 0}$ almost surely converges to a random probability 
measure $\Gamma_\infty$ if and only if,  for every continuous function $f \colon A \to \RR$, $\Gamma_k(f)$ almost surely converges to 
$\Gamma_\infty(f)$.  
\end{lemma}

\begin{proof}
The "only if" part is obvious. Conversely, if  for each continuous function $f \colon A \to \RR$, $\Gamma_k(f)$ almost surely converges to 
$\Gamma_\infty(f)$, then the full set of convergence can be taken independently of $f$ 
 by using the separability of the space of continuous functions on $A$. This shows the almost sure weak convergence $\Gamma_k \to \Gamma_\infty$ 
(see~\cite{Crauel} or~\cite{BPR} for details).  
\end{proof}

Recall that $\rho'$ denotes the  Kantorovich metric (defined by~\eqref{eq:def_Kantorovich}) 
 induced by $\rho$. 

\begin{lemma}\label{lemma:degenerate}
Let $\FF$ be a filtration and $X$ an 
$\FF_0$-measurable random variable taking its values in a compact metric space 
$(A,\rho)$. 
Then one always has the almost sure convergence as well as the 
$L^1$-convergence $\LL(X\given\FF_n) \to \LL(X\given\FF_{-\infty})$, 
i.e.
$$ \rho' \bigl(\LL(X\given\FF_n),\LL(X\given\FF_{-\infty})\bigr) \tend{n}{-\infty}0
\quad \text{almost surely}$$
and
$$\EE\Bigl[ \rho' \bigl(\LL(X\given\FF_n),\LL(X\given\FF_{-\infty})\bigr) \Bigr] \tend{n}{-\infty}0.$$
\end{lemma}

\begin{proof}
By L\'evy's reversed martingale convergence theorem, the convergence 
$\EE\bigl[f(X)\given\FF_n\bigr] \to \EE\bigl[f(X)\given\FF_{-\infty}\bigr]$  
holds almost surely for every 
continuous functions $f\colon A \to \RR$. 
The almost sure weak convergence $\LL(X\given\FF_n) \to \LL(X\given\FF_{-\infty})$ 
follows from Lemma \ref{lemma:convergence_conditional_law}. 
Since the Kantorovich distance metrizes the weak convergence, we get 
the almost sure convergence of 
$\rho' \bigl(\LL(X\given\FF_n),\LL(X\given\FF_{-\infty})\bigr)$ to $0$, as well as  
the $L^1$-convergence by the dominated convergence theorem.
\end{proof}

\begin{example*}[{\bf Poissonian Markov chain}]
Consider Example~\ref{exple:poisson}. 
We are going to determine the conditional law $\LL(X_0\given\FF_{-\infty})$. 
For every $n \leq -1$, the conditional law $\LL(X_0\given\FF_{n})$ is the 
binomial distribution on $\{0, \ldots, X_n\}$ with success probability parameter 
$\theta_n:=\lambda_0/\lambda_n$. 
Since ${(X_n)}_{n \leq 0}$ is decreasing, $X_n$ almost surely goes to a 
random variable $X_{-\infty}$ takings its values in $\N \cup \{+\infty\}$. 
\begin{itemize}
	\item \emph{Case 1: $\lambda_n \to \lambda_{-\infty} < \infty$.} 
In this case, it is easy to see with the help of Fourier transforms that 
$X_{-\infty}$ has the Poisson distribution with mean 
$\lambda_{-\infty}$. And by Lemma~\ref{lemma:degenerate}, 
$\LL(X_0\given\FF_{-\infty})$ is the 
binomial distribution on $\{0, \ldots, X_{-\infty}\}$ with success probability parameter 
$\lambda_0/\lambda_{-\infty}$. 
	
	\item \emph{Case 2: $\lambda_n \to +\infty$.} 
In this case, $X_n$ almost surely goes to $+\infty$. 
Indeed, $\PP(X_{-\infty}>K)\ge \PP(X_n > K) \to 1$ for any $K>0$. 
By the well-known Poisson approximation to the binomial distribution, 
it is expected that $\LL(X_0\given\FF_{n})$ should be well approximated by 
the Poisson distribution with mean $X_n\theta_n$ and then that $\LL(X_0\given\FF_{-\infty})$ 
should be the deterministic Poisson distribution with mean $\lambda_0$ 
(that is, the law of $X_0$). 
We prove it using Lemma~\ref{lemma:degenerate}. Let $\LL_n:=\LL(X_0\given\FF_n)$, 
denote by $\P(\lambda)$ the Poisson distribution with mean $\lambda$ 
and by $\Bin(k,\theta)$ the binomial distribution on $\{0, \ldots, k\}$ with 
success probability parameter $\theta$.  
Let $\rho$ be the discrete distance on the state space $\N$ of $X_0$. 
By introducing an appropriate coupling of $\P(\lambda)$ and $\Bin(k,\theta)$, as described in the introduction of~\cite{Lindvall}, it is not difficult to prove that 
$$\rho'\Bigl(\Bin(k,\theta), \P(k\theta)\Bigr) \leq k\theta^2.$$ 
By applying this result, 
$$\rho'\Bigl(\LL_n, \P(X_n\theta_n)\Bigr) \leq X_n\theta_n^2 =\frac{X_n}{\lambda_n}\frac{\lambda_0^2}{\lambda_n}.$$ 
Hence 
\begin{equation}
\label{Eq:cvgce1}
  \EE\left[\rho'\bigl(\LL_n, \P(X_n\theta_n)\bigr)\right] \to 0.
\end{equation}
On the other hand, for every $\lambda\geq\lambda'>0$, using the fact that  $\P(\lambda)=\P(\lambda')\ast\P(\lambda-\lambda')$, it is easy to derive the 
inequality 
$$\rho'\Bigl(\P(\lambda), \P(\lambda')\Bigr) \leq 1 - \exp(\lambda'-\lambda)\le |\lambda - \lambda'|.$$
Thus
$$\rho'\Bigl(\P(X_n\theta_n), \P(\lambda_0)\Bigr) \leq |X_n\theta_n-\lambda_0|.$$
Since $\var(X_n\theta_n)=\theta_n^2\lambda_n=\lambda_0^2/\lambda_n\to 0$, we get by Tchebychev's inequality, 
$X_n\theta_n \to \lambda_0$ in probability, which implies that 
$$
\EE\left[\rho'\Bigl(\P(X_n\theta_n), \P(\lambda_0)\Bigr) \right] \to 0.
$$
Together with~\eqref{Eq:cvgce1}, this yields 
$$\EE\left[\rho'\bigl(\LL_n, \P(\lambda_0)\bigr)\right] \to 0.$$
Comparing with Lemma~\ref{lemma:degenerate}, we get, as expected,  
$\LL(X_0\given\FF_{-\infty})=\P(\lambda_0)$.
\end{itemize}
The second assertion of Theorem~\ref{thm:monotonic} shows that the Poissonian Markov chain  
generates a standard filtration when $\lambda_n \to +\infty$, and a non-Kolmogorovian filtration otherwise.
\end{example*}

\subsubsection{Tool 2: Ordered couplings and linear metrics}

\begin{lemma}\label{lemma:transitivity}
 Let $\mu$, $\nu$ and $\eta$ be probability measures defined on an ordered set $E$ such that $\mu\leqst\nu$ and $\nu\leqst\eta$. 
 Then we can find three random variables  $X_\mu$, $X_\nu$, $X_\eta$ on the same probability space, with respective distribution $\mu$, $\nu$ and $\eta$, such that $X_\mu\le X_\nu\le X_\eta$ a.s. In particular, $\mu\leqst\eta$.
\end{lemma}

\begin{proof}
Let us consider three copies $E_1, E_2, E_3$ of $E$. 
Since $\mu\leqst\nu$, we can find a probability measure $\PP_{\mu, \nu}$ on $E_1\times E_2$ which is a coupling of $\mu$ and $\nu$, such that 
$\PP_{\mu, \nu}\left( \{(x_1,x_2): x_1\le x_2\}\right)=1$.
In the same way, we can find a probability measure $\PP_{\nu, \eta}$ on $E_2\times E_3$ which is a coupling of $\nu$ and $\eta$, such that 
$\PP_{\nu, \eta}\left( \{(x_2,x_3): x_2\le x_3\}\right)=1$. 
We consider the relatively independent coupling of $\PP_{\mu, \nu}$ and $\PP_{\nu, \eta}$ over $E_2$, which is the probability measure on $E_1\times E_2\times E_3$, defined by 
$$
\PP(A\times B\times C):= \int_B d\nu(x)\ \PP_{\mu, \nu}(A\times E_2|x_2=x)\ \PP_{\nu, \eta}(E_2\times C|x_2=x).
$$
Under $\PP$, the pair $(x_1,x_2)$ follows $\PP_{\mu, \nu}$ and the pair $(x_2,x_3)$ follows $\PP_{\nu, \eta}$. In particular, $x_1, x_2$ and $x_3$ are respectively distributed according to $\mu, \nu$ and $\eta$, and we have 
$$\PP\left( \{(x_1,x_2,x_3): x_1\le x_2\le x_3\}\right)=1.$$
\end{proof}

\begin{definition}
A pseudometric on an ordered set is \emph{linear} if 
$\rho(a,c) = \rho(a,b) + \rho(b,c)$ for every $a \leq b \leq c$. 
\end{definition} 

\begin{lemma}
  \label{lemma:linear-Kanto}
  Let $\rho$ be a linear pseudometric on a totally ordered set $A$, and let $\rho'$ be the associated Kantorovich pseudometric on the set of probability measures on $A$.
  Let $(Y_\mu,Y_\nu)$ be an ordered coupling of two probability measures $\mu$ and $\nu$ on $A$.
  Then
  $$\EE[\rho(Y_\mu,Y_\nu)]=\rho'(\mu,\nu). $$
In other words, the Kantorovich distance is achieved by any ordered coupling. 

Moreover, the Kantorovich pseudometric $\rho'$ is linear for the stochastic order: if $\mu\leqst\nu\leqst\eta$, one has 
 \begin{equation}
   \label{eq:linearity-Kanto}
 \rho'(\mu,\eta) =  \rho'(\mu,\nu) +  \rho'(\nu,\eta).
 \end{equation} 
\end{lemma}

\begin{proof}
  Since $\rho$ is linear and the set is totally ordered, we can find a non-decreasing map $\varphi:A\to\RR$ such that for all $x,y\in A$, $\rho(x,y)=|\varphi(x)-\varphi(y)|$. Hence we can assume without loss of generality that $A\subset\RR$ and $\rho(x,y)=|x-y|$.
  Since $(Y_\mu,Y_\nu)$ is an ordered coupling, we can also assume that $Y_\mu\ge Y_\nu$ a.s. Thus,
  $$ \EE[\rho(Y_\mu,Y_\nu)] = \EE [Y_\mu]-\EE[Y_\nu]\ge 0. $$
  Now, consider any coupling $(X_\mu,X_\nu)$ of $\mu$ and $\nu$. Then
  $$ \EE[\rho(X_\mu,X_\nu)] = \EE [|X_\mu-X_\nu|] \ge  \Bigl|\EE [X_\mu-X_\nu]\Bigr|
  = \Bigl|\EE [X_\mu]-\EE [X_\nu]\Bigr| = \EE[\rho(Y_\mu,Y_\nu)], $$
  which proves the first assertion of the lemma.
  
  Now, assuming that $\mu\leqst\nu\leqst\eta$, we consider an ordered coupling $(Y_\mu,Y_\nu,Y_\eta)$ where $Y_\mu\le Y_\nu\le Y_\eta$ a.s (see Lemma~\ref{lemma:transitivity}). Then,  
  $$ \rho'(\mu,\eta) = \EE[\rho(Y_\mu,Y_\eta)] = \EE[\rho(Y_\mu,Y_\nu)] + \EE[\rho(Y_\nu,Y_\eta)] =  \rho'(\mu,\nu) +  \rho'(\nu,\eta),$$
and the proof is over.
\end{proof}

In the next proposition, ${(X_n)}_{n \leq 0}$ is a monotonic Markov process 
with a given increasing representation $(f_n, U_n)$ (see Section~\ref{sec:monotonicrepresentation}), and we assume that all the state spaces $A_n$ are totally ordered.
Given a distance $\rho_0$ on $A_0$, we iteratively define the pseudometrics $\rho_n$ on $A_n$ as in Section~\ref{sec:vershik}. 
As explained in Section~\ref{sec:PWcoupling}, for any $m\le 0$, for any $x_m\in A_m$, we denote by $(Y_n(m, x_m))_{m\le n\le 0}$ the Propp-Wilson coupling starting at $x_m$.

This proposition is the main point in the demonstration 
of Theorem~\ref{thm:monotonic}. 
It will also be used later to derive the intrinsic metrics on the 
Pascal and Euler graphs.

\begin{ppsition}\label{ppsition:KantoMarkov}
Assume that $\rho_0$ is a linear distance on $A_0$.  Then 
for all $n\le 0$, $\rho_n$ is a linear pseudometric on $A_n$. Moreover, for all $(y,z)$ in $A_n$, $\rho_n(y,z)$ is the Kantorovich distance between 
$\LL(X_0 \given X_n=y)$ and $\LL(X_0 \given X_n=z)$ 
induced by $\rho_0$ and 
 $$ \forall y,z,\quad \rho_n(y,z)= \EE\Bigl[ \rho_0\left( Y_0(n,y),Y_0(n,z)\right)\Bigr].$$ 
\end{ppsition}

\begin{proof}
The statement of the lemma obviously holds for $n=0$. Assume that it holds for $n+1$ ($n\le -1$). 
 Since the updating functions $f_n$ are increasing, for all $(y,z)$ in $A_n$, 
 the random pair $\Bigl( Y_{n+1}(n,y),Y_{n+1}(n,z)\Bigr)$ is an ordered coupling of $\LL(X_{n+1} \given X_n=y)$ and $\LL(X_{n+1} \given X_n=z)$.    
 Therefore  by Lemma~\ref{lemma:linear-Kanto} and using the linearity of $\rho_{n+1}$,
 $$  \rho_n(y,z)  := (\rho_{n+1})'\Bigl({\cal L}(X_{n+1} \given X_n=y), {\cal L}(X_{n+1} \given X_n=z)\Bigr)$$
 is a linear distance, and moreover
 $$
 \rho_n(y,z) =
\EE\Bigl[ \rho_{n+1}\Bigl( Y_{n+1}(n,y),Y_{n+1}(n,z)\Bigr) \Bigr]. $$ 
By induction, this is equal to
$$ \EE\Bigl[ \rho_0\Bigl(    Y_0( n+1,Y_{n+1}(n,y) )  ,  Y_0( n+1,Y_{n+1}(n,z) )   \Bigr)\Bigr].$$
Observe now that for any $x$, we have $Y_0\Bigl( n+1, Y_{n+1}(n,x)\Bigr)=Y_0(n,x)$. Hence, 
$$\rho_n(y,z)= \EE\Bigl[ \rho_0\Bigl( Y_0(n,y),Y_0(n,z)\Bigr)\Bigr].$$
Moreover, the random pair $\Bigl( Y_0(n,y),Y_0(n,z)\Bigr)$ is an ordered coupling of $\LL(X_0 \given X_n=y)$ and $\LL(X_0 \given X_n=z)$. 
Therefore, by Lemma~\ref{lemma:linear-Kanto}, since $\rho_0$ is linear, we get that $\rho_n(y,z)$ is the Kantorovich distance between 
$\LL(X_0 \given X_n=y)$ and $\LL(X_0 \given X_n=z)$ 
induced by $\rho_0$.
\end{proof}

\subsection{Proof of Theorem~\ref{thm:monotonic}}

We are now ready to prove the equivalence between the conditions stated in Theorem~\ref{thm:monotonic}. 

We have seen at the end of Section~\ref{sec:monotonicrepresentation} that there exists an increasing representation, thus $\textit{(c)} \implies \textit{(b)}$ is obvious.
$\textit{(b)} \implies \textit{(a)}$ stems from Lemma~\ref{lemma:generatingparam}. 
$\textit{(a)} \implies \textit{(e)}$ is obvious  and $\textit{(e)} \implies \textit{(d)}$ 
stems from Lemma~\ref{lemma:degenerate}.
The main point to show is $\textit{(d)} \implies \textit{(c)}$. 
Let  ${(f_n, U_n)}_{n \leq 0}$ be a parameterization of ${(X_n)}_{n \leq 0}$ 
with increasing updating functions $f_n$. We denote by $\rho$ the usual distance 
 on $\RR$. 

By hypothesis, for each fixed $n\le 0$, $\LL(X_n\given\FF_{-\infty})=\LL(X_n)$. 
Without loss of generality, we can assume that every $X_n$ takes its values in a compact 
subset of $\RR$. 
 Lemma~\ref{lemma:degenerate} then gives the $L^1$-convergence of 
 $\LL(X_n\given\FF_{m})$ to $\LL(X_n)$ as $m$ goes to $-\infty$: 
$$ s_m \egdef \EE\Bigl[ \rho'\bigl(\LL(X_n\given\FF_{m}),\LL(X_n)\bigr)\Bigr]\tend{m}{-\infty} 0. $$ 
Hence, for each $m$ there exists $x_m$ in the state space of $X_m$ such that 
$$ \rho'\bigl(\LL(X_n\given X_{m}=x_m),\LL(X_n)\bigr)\le s_m. $$ 
Consider the Propp-Wilson coupling construction of Section~\ref{sec:PWcoupling}. 
Since $\rho$ is a linear distance, and each $f_n$ is increasing, we can apply Lemma~\ref{lemma:linear-Kanto} to get
$$\EE\left[\rho\left(X_n, Y_n(m,x_m)\right) \given \FF_{m}\right] 
= \rho'\Bigl(\LL(X_n\given \FF_{m}),\LL(X_n\given X_{m}=x_{m}) \Bigr)$$
for every integer $m<n\le 0$. Taking the expectation on both sides yields
\begin{align*}
  \EE\left[\rho\left(X_n, Y_n(m,x_m)\right)  \right] 
&=\EE\left[ \rho'\Bigl(\LL(X_n\given \FF_{m}),\LL(X_n\given \FF_{m}=x_{m}) \Bigr) \right] \\
&\le \EE\left[ \rho'\Bigl(\LL(X_n\given \FF_{m}),\LL(X_n) \Bigr) \right] +
\EE\left[ \rho'\Bigl(\LL(X_n),\LL(X_n\given \FF_{m}=x_{m}) \Bigr) \right] \\
& \le 2 s_m \tend{m}{-\infty}0.
\end{align*}
Then $\textit{(c)}$ follows from Lemma~\ref{lemma:PWgenerating}. 

\medskip 
Now to prove {\it 2)} we take the sequence $(x_m)$ for $n=0$ and we use again the Propp-Wilson coupling. 
Assuming that the Markov process is identifiable, the iterated Kantorovich pseudometrics $\rho_n$ introduced in Section~\ref{sec:vershik} with initial distance $\rho_0=\rho$ are metrics.

By Proposition~\ref{ppsition:KantoMarkov}, for every integer $m\le n \le0$,
$$\rho_n\bigl(X_n,Y_n(m,x_m)\bigr)=
\EE\left[\rho_0\bigl(X_0, Y_0(m,x_m)\bigr) \bigm\vert X_n,Y_n(m,x_m)\right] .$$ 
We have seen in the first part of the proof that 
$$ \EE\left[\rho_0\bigl(X_0, Y_0(m,x_m)\bigr)\right] \tend{m}{-\infty}0 $$ 
 under the 
assumption $\LL(X_0\given\FF_{-\infty})=\LL(X_0)$. 
Thus, for every $n \leq 0$, 
the expectation $\EE\left[\rho_n\bigl(X_n, Y_n(m,x_m)\bigr)\right]$ 
 goes to 0 as $m\to-\infty$, and Lemma~\ref{lemma:PWgenerating} gives the result. 


\section{Multidimensional monotonic Markov processes}\label{sec:multimarkov}

We now want to prove a multidimensional version of 
Theorem~\ref{thm:monotonic}. 
However, as compared to the unidimensional case, the criterion we obtain only guarantee standardness 
of the filtration, but not the existence of a generating parameterization. 
In this section, ${(X_n)}_{n\le 0}$ is a Markov process taking its values in $\RR^d$ for some integer $d\ge1$ or $d=\infty$.  
For each $n\le 0$, we denote by $\mu_n$ the law of $X_n$, and by $A_n$ the support of $\mu_n$. 

\subsection{Monotonicity for multidimensional Markov processes}
We first have to extend the notion of monotonicity given in Definition~\ref{def:monotonic} to the case of
multidimensional Markov processes.

\begin{definition}\label{def:monotonic-multidim}
We say that $(X_n)$ is \emph{monotonic} if for each $n<0$, for all $x, x'$ in $A_n$, there exists a coupling $(Y,Y')$ of $\LL(X_{n+1}\given X_n=x)$ and $\LL(X_{n+1}\given X_n=x')$, whose distribution depends measurably on $(x,x')$, and which is \emph{well-ordered with respect to $(x,x')$}, which means that, for each $1\le k\le d$,
\begin{itemize}
  \item $x(k)\le x'(k)\Longrightarrow \PP\Bigl(Y(k)\le Y'(k)\Bigr)=1$,
  \item $x(k)\ge x'(k)\Longrightarrow \PP\Bigl(Y(k)\ge Y'(k)\Bigr)=1$.
\end{itemize}

\end{definition}

For example, ${(X_n)}_{n \leq 0}$ is a monotonic Markov process when 
the one-dimensional coordinate processes ${\big(X_n(k)\bigr)}_{n \leq 0}$ are independent 
monotonic Markov processes. 
But the definition does not require nor imply that the coordinate processes 
${\big(X_n(k)\bigr)}_{n \leq 0}$ are Markovian. 

Theorem \ref{thm:monotonic} will be generalized to $\RR^d$-valued monotonic 
processes in Theorem \ref{thm:monotonicmulti}, except that we will not get 
the simpler criteria 2) under the identifiability assumption. 
This will be obtained with the help of Vershik's criterion 
(Lemma~\ref{lemma:VershikMarkov}) in Theorem \ref{thm:fullymonotonicmulti} for \emph{strongly 
monotonic} Markov processes, defined below.

\begin{definition}\label{def:fullymonotonic}
A Markov process ${(X_n)}_{n \leq 0}$ taking its values in 
$\RR^d$ is said to be \emph{strongly monotonic} if it is monotonic in the sense of 
the previous definition and if in addition, denoting by $\FF$ the filtration it generates and by 
 $\FF(k)$ 
the filtration generated by the $k$-th coordinate process 
${\big(X_n(k)\bigr)}_{n \leq 0}$, the two following conditions hold:
\begin{enumerate}[{\rm a)}]
\item each process  ${\big(X_n(k)\bigr)}_{n \leq 0}$ is Markovian,
\item each filtration $\FF(k)$ is immersed in the filtration $\FF$,
\end{enumerate}
\end{definition}

Note that conditions a) and b) together mean that 
each process  ${\big(X_n(k)\bigr)}_{n \leq 0}$ is Markovian  
with respect to $\FF$. 

The proof of the following lemma is left to the reader.
\begin{lemma}\label{lemma:fullymonotonic}
Let  ${(X_n)}_{n \leq 0}$ be a strongly monotonic Markov process taking its values in 
$\RR^d$. Then each coordinate process 
${\big(X_n(k)\bigr)}_{n \leq 0}$ is a monotonic Markov process. 
\end{lemma}

The converse of Lemma~\ref{lemma:fullymonotonic} is false, as shown by the example below. 

\begin{example}[\emph{Random walk on a square}]\label{example:square}
Let ${(X_n)}_{n \leq 0}$ be the stationary random walk on the square 
$\{-1,1\}\times\{-1,1\}$, whose distribution is defined by:
\begin{itemize}
	\item \emph{(Instantaneous laws)} each $X_n$ has the uniform distribution 
	on $\{-1,1\}\times\{-1,1\}$;

	\item  \emph{(Markovian transition)} at each time, the process jumps at random 
	from one vertex of the square to one of its two connected vertices, more precisely, 
	 given $X_n=\bigl(x_n(1), x_n(2)\bigr)$, 
	the random variable $X_{n+1}$ takes either the value $\bigl(-x_n(1), x_n(2)\bigr)$ 
	or $\bigl(x_n(1), -x_n(2)\bigr)$ with equal probability. 
\end{itemize}
Each of the two coordinate processes ${\big(X_n(1)\bigr)}_{n \leq 0}$ and 
${\big(X_n(2)\bigr)}_{n \leq 0}$ is a sequence of independent random 
variables, therefore is a monotonic Markov process. It is not difficult to see 
in addition that each of them is Markovian with respect to the filtration $\FF$ of 
${(X_n)}_{n \leq 0}$, hence the two conditions of Lemma~\ref{lemma:fullymonotonic} 
hold true. 
But one can easily check   that the process $(X_n)$ does not satisfy the conditions of Definition~\ref{def:monotonic-multidim}.

Note that the tail $\sigma$-field $\FF_{-\infty}$ is not degenerate because of the 
periodicity of  ${(X_n)}_{n \leq 0}$, hence we obviously know that standardness does not hold 
for $\FF$. 
\end{example}

\subsection{Standardness for monotonic multidimensional Markov processes}\label{sec:monotonicmulti}


Since we are interested in the filtration generated by ${(X_n)}_{n\le 0}$, one can assume without loss 
of generality that the support $A_n$ of the law of $X_n$ is included in $[0,1]^d$ for every $n \leq 0$.  
Indeed, applying a strictly increasing transformation on 
each coordinate of the process alters neither the Markov and the monotonicity properties, nor 
the $\sigma$-fields $\sigma(X_n)$. 

\begin{thm}\label{thm:monotonicmulti}
Let  ${(X_n)}_{n \leq 0}$ be a $d$-dimensional monotonic Markov process, and $\FF$ the filration it generates.
The following conditions are equivalent. 
	\begin{enumerate}[{(a)}]
	\item $\FF$ is Kolmogorovian.
	
	\item For every $n \leq 0$, the conditional law ${\cal L}(X_{n} \given \FF_{-\infty})$ is almost surely equal to ${\cal L}(X_{n})$.
	
	\item $\FF$ is standard.
	
	\end{enumerate}
\end{thm}

\begin{proof}
  We only have to prove that \textit{(b)} implies \textit{(c)}.  
  
  We consider a family ${(U_n^j)}_{n\le 0,j\ge1}$ of independent random variables, uniformly distributed on $[0,1]$. The standardness of the filtration generated by ${(X_n)}_{n\le0}$ will be proved by constructing a copy ${(Z_n)}_{n\le0}$ of ${(X_n)}_{n\le0}$ such that 
  \begin{itemize}
    \item For each $n\le0$, $Z_n$ is measurable with respect to the $\sigma$-algebra $\UU_n$ generated by ${(U_m^j)}_{m\le n,j\ge1}$. (Observe that the filtration $\UU\egdef{(\UU_n)}_{n\le0}$ is of product type.)
    \item The filtration generated by ${(Z_n)}_{n\le0}$ is immersed in $\UU$.
  \end{itemize}
For each $j\ge1$,  using the random variables $U_n^j$ 
we will construct inductively a process $Z^j\egdef {(Z_n^j)}_{n_j\le n\le 0}$, 
where ${(n_j)}_{j \geq 1}$ is a decreasing sequence of negative integers to be precised later.  
Each $Z_n$ will then be obtained as an almost-sure limit, as $j\to\infty$, of the sequence ${(Z_n^j)}$.

\subsubsection*{Construction of a sequence of processes}

We consider as in Section \ref{sec:markovstandard} that the Markovian transitions 
are given by kernels $P_n$.
For every $n <0$, we take an updating function $f_{n+1}:A_n\times[0,1]\to A_{n+1}$ 
such that $\LL\bigl(f_n(x,U)\bigr) =P_n(x, \cdot)$ for every $x \in A_n$ 
whenever $U$ is uniformly distributed on $[0,1]$. 
 
To construct the first process $Z^1$, we choose an appropriate point $x_{n_1}\in A_{n_1}$ (which is also to be precised later), and set $Z_{n_1}^1\egdef x_{n_1}$. Then for $n_1\le n<0$, we inductively define
\[  Z_{n+1}^1\egdef f_{n+1}\left(Z_{n}^1,U_{n+1}^1\right),\]
so that
\[ \LL\left( Z^1 \right)  = \LL\left({(X_n)}_{n_1\le n\le 0}\given X_{n_1}=x_{n_1}\right).\]

\begin{figure}
\begin{center}  \includegraphics{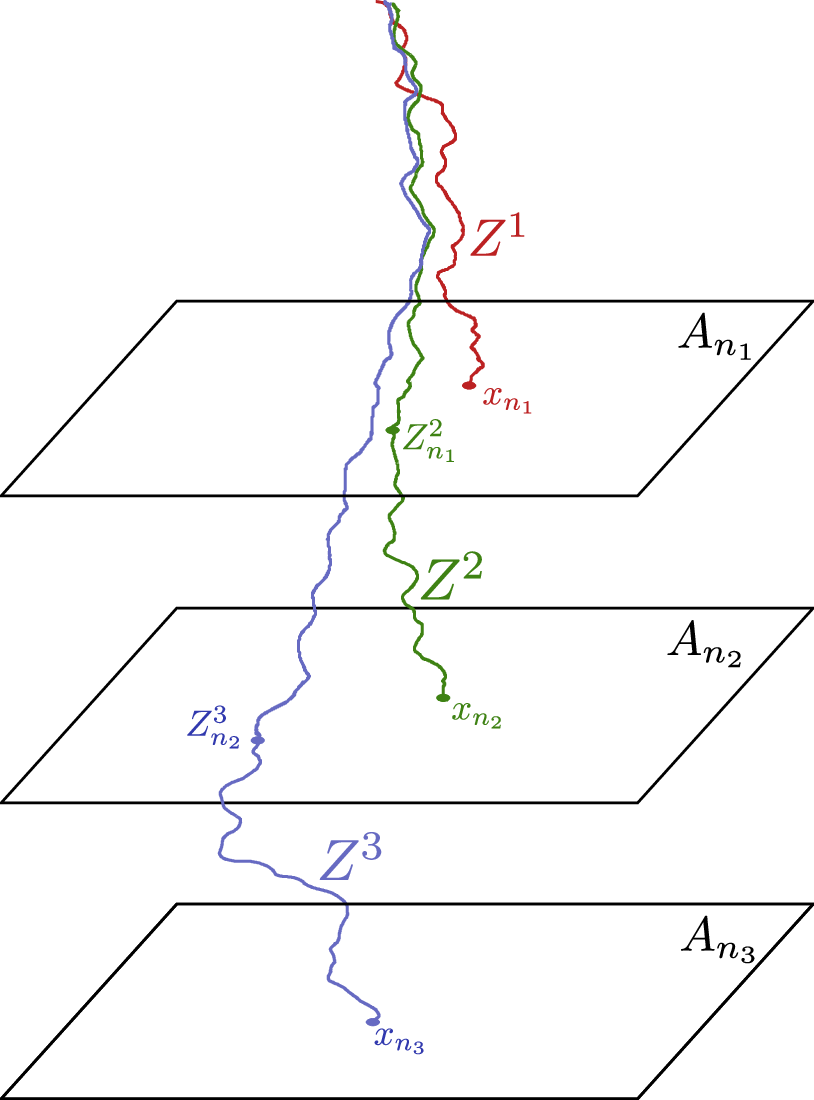}  \end{center}
\caption{Construction of the sequence of processes $(Z^j)$. The processes $Z^j$ and $Z^{j+1}$ are coupled in a well-ordered way from time $n_j$ to 0.}
  \label{Fig:Z}
\end{figure}

Assume that we have constructed the processes $Z^i$ for all $1\le i\le j$. Then we get the process $Z^{j+1}$ by choosing an appropriate point $x_{n_{j+1}}\in A_{n_{j+1}}$, setting $Z_{n_{j+1}}^{j+1}\egdef x_{n_{j+1}}$, and inductively
\[  Z_{n+1}^{j+1}\egdef
\begin{cases}
  f_{n+1}\left(Z_{n}^{j+1},U_{n+1}^{j+1}\right) & \text{ for }n_{j+1}\le n<n_j,\\
  f^{j+1}_{n+1}\left(Z_{n}^{j},Z_{n}^{j+1},Z_{n+1}^{j},U_{n+1}^{j+1}\right) & \text{ for }n_{j}\le n<0.
\end{cases}
\]
where the function $f^{j+1}_{n+1}$ is recursively obtained as follows. 
Let $(Z, Z', Z_+, Z'_+)$ be a random four-tuple such that 
$\LL(Z, Z') =\LL\left(Z_{n}^{j}, Z_{n}^{j+1}\right)$ and 
$\LL\Bigl((Z_+, Z'_+) \given Z, Z'\Bigr) = \Lambda_{Z_1, Z'_1}$,
where $\Lambda_{x,x'}$ is the well-ordered coupling of Definition \ref{def:monotonic-multidim}. 
Recall that the first and second margins of $\Lambda_{x,x'}$  
are $P_{n+1}(x,\cdot)$ and $P_{n+1}(x',\cdot)$.
Now, consider a kernel $Q$ being a regular version of the conditional distribution 
$\LL(Z'_+ \given Z, Z', Z_+)$, and then take $f^{j+1}_{n+1}$ such that 
$\LL\left(f^{j+1}_{n+1}(z, z', z_+, U)\right) = Q\bigl((z, z', z_+), \cdot\bigr)$ 
for every $(z, z', z_+) \in A_{n} \times A_n \times A_{n+1}$ whenever 
$U$ is uniformly distributed on $[0,1]$. 

In this way, we get by construction
\begin{equation}
  \label{eq:noel} \forall n_j\le n<0, \quad \LL\Bigl(Z_{n+1}^{j+1} \given Z_{n}^{j}, Z_{n}^{j+1}\Bigr) 
= P_{n+1}(Z_{n}^{j+1},\cdot).
\end{equation}
Moreover, we easily prove by induction that $Z_n^j$ is measurable with respect to 
$\sigma(U_{m,i}; m \leq n, 1 \leq i \leq j) \subset {\cal U}_n$ 
for all possible $n$ and $j$.
Now, we want to prove that, for all $j\ge1$ and all $n_j\le n<0$,
\begin{equation}\label{eq:immersion}
\LL(Z_{n+1}^{j} \given {\cal U}_n) 
= P_{n+1}(Z_{n}^{j},\cdot).
\end{equation}
This equality stems from the definition of $f_{n+1}$ for $j=1$. 
Assuming the equality holds for $j$, we show that it holds for $j+1$ as follows. 
When $n_{j+1} \leq n < n_{j}$, this comes again from the definition of $f_{n+1}$.
If $n_{j} \leq n < 0$,  since $Z_{n+1}^{j+1} = f^{j+1}_{n+1}\left(Z_{n}^{j},Z_{n}^{j+1},Z_{n+1}^{j},U_{n+1}^{j+1}\right) $, where $U_{n+1}^{j+1}$ is independent of $(Z_{n}^{j},Z_{n}^{j+1},Z_{n+1}^{j})$, we get
$$\LL(Z_{n+1}^{j+1} \given {\cal U}_n \vee Z_{n+1}^{j}) = 
\LL(Z_{n+1}^{j+1} \given Z_n^{j}, Z_n^{j+1}, Z_{n+1}^{j}).$$
Using the induction hypothesis, we know that $\LL(Z_{n+1}^{j} \given {\cal U}_n) = \LL(Z_{n+1}^{j} \given Z_n^j)$, and we can write
$$\LL(Z_{n+1}^{j+1} \given {\cal U}_n) = 
\LL(Z_{n+1}^{j+1} \given Z_n^{j}, Z_n^{j+1}).$$
Recalling~\eqref{eq:noel}, we conclude that~\eqref{eq:immersion} holds for $j+1$.

{From}~\eqref{eq:immersion}, it follows that
\[ \LL\left( Z^{j} \right)  = \LL\left({(X_n)}_{n_{j}\le n\le 0}\given X_{n_{j}}=x_{n_{j}}\right)\]
for every $j \geq 1$. 
Moreover, given $Z_{n_j}^{j+1}$, the processes $Z^j$ and $Z^{j+1}$ are coupled from $n_j$ in a well-ordered way with respect to $\left(x_{n_j},Z_{n_j}^{j+1}\right)$. (See Figure~\ref{Fig:Z}.)

\subsubsection*{Choice of the sequences $(n_j)$ and $(x_{n_j})$}

In this part we explain how we can choose the sequences $(n_j)$ and $(x_{n_j})$ so that 
\begin{equation}
  \label{eq:cond1}
  \forall n\le0,\ Z_n^j\text{ converges almost surely as }j\to \infty.
\end{equation}
Moreover, to ensure that the filtration generated by the limit process ${(Z_n)}_{n\le0}$ is immersed in $\UU$, we will also require the following convergence:
\begin{equation}
  \label{eq:cond2}
  \forall n\le-1,\ \LL\left(Z_{n+1}^j\given Z_n^j\right) \tend[a.s.]{j}{\infty}\LL\left(Z_{n+1}\given Z_n\right).
\end{equation}

Recall we assumed that $A_n\subset[0,1]^d$. Let us define the distance $\rho$ on $A_n$ by
$\rho(x,x')\egdef\sum_{k=1}^d a_k|x(k) - x'(k)|$, 
where, in order to handle the case when $d=\infty$, we take a sequence 
${(a_k)}_{k=1}^d$ of positive numbers satisfying $\sum a_k =1$. 
For any $j\ge1$, we also define the distance $\Delta_j$ on ${(\RR^d)}^j$ by 
\[
  \Delta_j\bigl((x_1,\ldots,x_j),(y_1,\ldots,y_j)\bigr) \egdef \max_{1 \leq \ell \leq j}\rho(x_\ell,y_\ell).
\]
Let us introduce, for $j\ge 1$,  and $\ell\le -j$, the measurable subset of $A_\ell$
\[
  M_\ell^j \egdef \left\{x\in A_{\ell}: \Delta_j'\left(
\LL\Bigl({(X_n)}_{-j<n\le0}\given X_{\ell}=x\Bigr),\LL\Bigl({(X_n)}_{-j<n\le0}\Bigr)\right) > 2^{-j}\right\}.
\]

Applying Lemma~\ref{lemma:degenerate} and using hypothesis \textit{(b)} 
of Theorem \ref{thm:monotonicmulti}, for each $j\ge1$, 
\begin{equation}
  \label{eq:mu_Mlj}
  \mu_{\ell}\left(M_\ell^j\right)\tend{\ell}{-\infty}0.
\end{equation}

For each $n\le -1$, we denote by $\MM_1(A_{n+1})$ the set of probability measures on $A_{n+1}$, equipped with the Kantorovich distance $\rho'$.
We also consider $\varphi_n:A_n\to \MM_1(A_{n+1})$, defined by
\[
  \varphi_n(z)\egdef \LL\left(X_{n+1}\given X_n=z\right).
\]
Since $\varphi_n$ is a measurable function, we can apply Lusin Theorem to get the existence, for any $k\ge1$, of a continuous approximation $\varphi_n^k$ of $\varphi_n$, such that
\begin{equation}
\label{eq:Lusin}
\mu_n\left(\varphi_n\neq\varphi_n^k\right) < 2^{-k} . 
\end{equation}

Let us choose $n_1$ and $x_{n_1}$: By~\eqref{eq:mu_Mlj}, we can choose $|n_1|$  large enough so that $\mu_{n_1}\left(M_{n_1}^1\right)<2^{-1}$, and then choose $x_{n_1}\in A_{n_1}\setminus M_{n_1}^1$.

Assume now that for some $j\ge2$ we have already chosen $n_{j-1}$ such that $\mu_{n_{j-1}}\left(M_{n_{j-1}}^{j-1}\right)<2^{-(j-1)}$ and 
$x_{n_{j-1}}\in A_{n_{j-1}}\setminus M_{n_{j-1}}^{j-1}$. 
By Lemma~\ref{lemma:degenerate} and using hypothesis \textit{(b)}, we get
\[
  \PP\left( X_{n_{j-1}}\in M_{n_{j-1}}^{j-1}\given X_\ell\right) \tend[a.s.]{\ell}{-\infty} \mu_{n_{j-1}}\left(M_{n_{j-1}}^{j-1}\right)<2^{-(j-1)},
\]
and for each $n,k$, $-j\le n\le 0$, $1\le k\le j$,
\[
  \PP\left( \varphi_n(X_n)\neq\varphi_n^k(X_n)\given X_\ell\right) \tend[a.s.]{\ell}{-\infty} \mu_n\left(\varphi_n\neq\varphi_n^k\right) < 2^{-k} . 
\]
Therefore, using also~\eqref{eq:mu_Mlj}, if $|n_j|$ is large enough, we will have
\[
  \mu_{n_{j}}\left(M_{n_{j}}^{j}\right)<2^{-j},
\]
and there exists $x_{n_j}\in A_{n_j}\setminus M_{n_{j}}^{j}$ such that
\begin{equation}
  \label{eq:inequality1}
  \PP\left( X_{n_{j-1}}\in M_{n_{j-1}}^{j-1}\given X_{n_j}=x_{n_j}\right) < 2^{-(j-1)},
\end{equation}
as well as 
\begin{equation}
  \label{eq:inequality2}
  \forall n,k,\ -j\le n\le 0,\ 1\le k\le j,\ \PP\left( \varphi_n(X_n)\neq\varphi_n^k(X_n)\given X_{n_j}=x_{n_j}\right)<2^{-k}.
\end{equation}

\subsubsection*{Convergence of the sequence of processes}
We want to prove that, for each $n\le0$, with the above choice of $(n_j)$ and $(x_{n_j})$,  the sequence ${(Z_n^j)}_{j\ge -n}$ is almost surely a Cauchy sequence.

Since we used well-ordered couplings in the construction of the processes $Z^j$, and since the distance $\delta$ defined by the absolute value on $\RR$ is linear, by application of Lemma~\ref{lemma:linear-Kanto},  we have, when $-j \leq n < 0$
\begin{align*}
   \EE\left[ \rho\left(Z_n^j,Z_n^{j+1}\right) \given Z_{n_j}^{j+1} \right] 
   & = \sum_{k=1}^d a_k\EE\left[ \left|Z_n^j(k) - Z_n^{j+1}(k)\right| \given Z_{n_j}^{j+1} \right] \\
   & = \sum_{k=1}^d a_k\delta' \Bigl( \LL\left(Z_n^j(k) \given Z_{n_j}^{j+1}\right), \LL\left(Z_n^{j+1}(k) \given Z_{n_j}^{j+1}\right)  \Bigr)\\
   & \leq \rho'\Bigl( \LL\left(Z_n^j \given Z_{n_j}^{j+1}\right), \LL\left(Z_n^{j+1} \given Z_{n_j}^{j+1}\right)  \Bigr)\\
   &= \rho'\left( \LL(X_n\given X_{n_j}=x_{n_j}),\LL(X_n\given X_{n_j}=Z_{n_j}^{j+1})\right),\numberthis \label{eq:well-ordered-coupling}
\end{align*}
the inequality coming from the fact that the minimum of a sum is larger than the sum of the minima. Note that, since the converse inequality is obvious by definition of the Kantorovich distance $\rho'$, the above inequality is in fact an equality. Then, by the triangular inequality, we can bound $\EE\left[ \rho\left(Z_n^j,Z_n^{j+1}\right) \given Z_{n_j}^{j+1} \right]$ by the sum
\begin{equation}
\label{eq:triangle}
  \rho'\left(\LL(X_n\given X_{n_j}=x_{n_j}),\LL(X_n)\right) 
  +  \rho'\left(\LL(X_n),\LL(X_n\given X_{n_j}=Z_{n_j}^{j+1})\right) .
\end{equation}
Recall we chose $x_{n_j}\in A_{n_j}\setminus M_{n_j}^j$, which ensures by definition of $M_{n_j}^j$ that the first term of~\eqref{eq:triangle} is bounded by $2^{-j}$. Moreover, the second term of~\eqref{eq:triangle} can be bounded by
\[
\indic_{Z_{n_j}^{j+1}\in M_{n_j}^j}+2^{-j}\indic_{Z_{n_j}^{j+1}\notin M_{n_j}^j}.
\]
By~\eqref{eq:inequality1}, for each $-j< n\le 0$,
\[
  \PP\left(Z_{n_j}^{j+1}\in M_{n_j}^j\right) = \PP\left(X_{n_j}\in M_{n_j}^j\given X_{n_{j+1}}=x_{n_{j+1}}\right) < 2^{-j}.
\]
Thus, by integrating with respect to $Z_{n_j}^{j+1}$, we obtain that $\EE\left[ \rho\left(Z_n^j,Z_n^{j+1}\right)  \right] $ is bounded above by 
\[
2^{-j}+2^{-j}+\PP\left(Z_{n_j}^{j+1}\in M_{n_j}^j\right)
\le 3 \times 2^{-j}.
\]
%
Therefore, for each fixed $n\le0$, ${(Z_n^j)}_{j>-n}$ is almost surely a Cauchy sequence and converges almost surely to some limit $Z_n$, which is measurable with respect to the $\sigma$-algebra $\UU_n$ generated by ${(U_m^j)}_{m\le n,j\ge1}$.

Observe that for any fixed $m\le 0$, since $x_{n_j}$ has been chosen in $A_{n_j}\setminus M_{n_j}^j$,
\[ 
  \LL\left({(Z_n^j)}_{m\le n\le 0}\right)=\LL\left({(X_n)}_{m\le n\le 0}\given X_{n_j}=x_{n_j}\right)\tend{j}{\infty} \LL\left({(X_n)}_{m\le n\le 0}\right).
\] 
Hence, we conclude that ${(Z_n)}_{n\le 0}$ is a copy of ${(X_n)}_{n\le 0}$.

\subsubsection*{Proof of the immersion of ${(Z_n)}_{n\le0}$ in $\UU$}

We need to prove that for all $n<0$, $\LL(Z_{n+1}\given \UU_n)=\LL(Z_{n+1}\given Z_n)$. 
We have already seen that 
\[
 \LL\left(Z_{n+1}^{j}\given \UU_n\right)= \LL\left(X_{n+1} \given X_n=Z_n^{j}\right) = \LL\left(Z_{n+1}^{j} \given Z_n^{j}\right).
\]
We now want to take the limit as $j\to\infty$. For any continuous function $g$ on $A_{n+1}$, we have 
$$
\EE\left[ g(Z_{n+1}^{j})\given \UU_n\right] \tend[a.s.]{j}{\infty} 
 \EE\left[g(Z_{n+1})\given \UU_n\right]
$$
by the conditional dominated convergence theorem. 
Therefore, 
by Lemma~\ref{lemma:convergence_conditional_law}, 
\[
  \LL(Z_{n+1}^{j}\given \UU_n)=\LL\left(Z_{n+1}^{j} \given Z_n^{j}\right)\tend[a.s.]{j}{\infty} \LL(Z_{n+1}\given \UU_n).
  \]
By the dominated convergence theorem, we then get
\begin{multline}
  \label{eq:convergence}\EE \left[ \rho' \left( \LL\left(Z_{n+1}^{j} \given Z_n^{j}\right) ,  \LL\left(Z_{n+1} \given Z_n\right) \right) \right] \\ \tend{j}{\infty} 
  \EE \left[ \rho' \left( \LL\left(Z_{n+1} \given  \UU_n\right) ,  \LL\left(Z_{n+1} \given Z_n\right) \right) \right].
\end{multline}
On the other hand, the LHS of the preceding formula can be rewritten as $\EE\left[\rho'\left(\varphi_n \left(Z_n^j\right),\varphi_n\left(Z_n\right)\right)\right]$, and bounded by the sum of the three following terms:
\begin{align*}
  T_1&\egdef \EE\left[\rho'\left(\varphi_n \left(Z_n^j\right),\varphi_n^k\left(Z_n^j\right)\right)\right],\\
  T_2&\egdef \EE\left[\rho'\left(\varphi_n^k \left(Z_n^j\right),\varphi_n^k\left(Z_n\right)\right)\right],\\
  T_3&\egdef \EE\left[\rho'\left(\varphi_n^k \left(Z_n\right),\varphi_n\left(Z_n\right)\right)\right].
\end{align*}
Using~\eqref{eq:Lusin}, $T_3\le 2^{-k}$ which can be made arbitrarily small by fixing $k$ large enough. Once $k$ has been fixed, $T_2\tend{j}{\infty}0$ 
by continuity of $\varphi_n^k$ and dominated convergence. 
Then, remembering~\eqref{eq:inequality2}, we get $T_1<2^{-k}$ as soon as $j\ge |n|$  and $j\ge k$. 
This proves that 
\[
  \EE\left[\rho'\left(\varphi_n \left(Z_n^j\right),\varphi_n\left(Z_n\right)\right)\right]\tend{j}{\infty}0.
\]
Comparing with~\eqref{eq:convergence}, we get the desired equality
\[
\LL\left(Z_{n+1}\given \UU_n\right)= \LL\left(Z_{n+1}\given Z_n\right).
\]
\end{proof}

\subsection{Computation of iterated Kantorovich metrics}\label{sec:multimonotonic_kantorovich}

Here we assume that  ${(X_n)}_{n \leq 0}$ is a \emph{strongly} monotonic Markov process
(Definition~\ref{def:fullymonotonic}). 
As before, we assume without loss of generality that it takes its values in $[0,1]^d$ 
equipped with the distance $\rho$ on ${[0,1]}^d$ defined by
$\rho(x,x')\egdef\sum_{k=1}^d a_k|x(k) - x'(k)|$, 
where ${(a_k)}_{k=1}^d$ is a sequence 
of positive numbers satisfying $\sum a_k =1$, whose role is 
to  handle the case when $d=\infty$. 

The purpose of this section is to establish a connection between the iterated Kantorovich metrics $\rho_n$ initiated by $\rho$ and those associated to the Markov processes ${\bigl(X_n(k)\bigr)}_{n \leq 0}$, initiated by the distance $\delta$ defined by the absolute value on $\RR$. 
Then, with the help of Vershik's criterion 
(Lemma \ref{lemma:VershikMarkov}), we will establish the analogue of criterion 2) 
in Theorem~\ref{thm:monotonic}. 

\begin{lemma}\label{lemma:kanto_mordered}  
For each $\ell\le0$, and each $n \in \{\ell, \ldots, 0\}$, 
\begin{multline*}
\rho'\left(\LL(X_n \given X_\ell=x_\ell), \LL(X_n \given X_\ell=x'_\ell)\right)  
\\
= \sum_{k=1}^d 
a_k \delta'\left(\LL\bigl(X_n(k) \given X_\ell(k)=x_\ell(k)\bigr), \LL\bigl(X_n(k) \given X_\ell(k)=x'_\ell(k)\bigr)\right) .
\end{multline*}
\end{lemma}

\begin{proof}
Let $x_\ell$ and $x'_\ell$ be two points in $A_\ell$. As in the proof of Theorem~\ref{thm:monotonicmulti}, we can construct two processes ${(Z_n)}_{n\ge\ell}$ and ${(Z'_n)}_{n\ge\ell}$ such that
\begin{itemize}
  \item $\LL\left({(Z_n)}_{n\ge\ell}\right) = \LL\left({(X_n)}_{n\ge\ell}\given X_\ell=x_\ell\right)$,
  \item $\LL\left({(Z'_n)}_{n\ge\ell}\right) = \LL\left({(X_n)}_{n\ge\ell}\given X_\ell=x'_\ell\right)$,
  \item for each $n\ge\ell$, the coupling $(Z_n,Z'_n)$ is well-ordered with respect to $(x_\ell,x'_\ell)$.
\end{itemize}
By similar arguments as those used in~\eqref{eq:well-ordered-coupling}, relying on 
Lemma~\ref{lemma:linear-Kanto}, we get
\begin{align*}
  \rho'\left(\LL(X_n \given X_\ell=x_\ell), \LL(X_n \given X_\ell=x'_\ell)\right)  
  & = \rho'\left(\LL(Z_n), \LL (Z'_n)\right)  \\
  & = \sum_{k=1}^d 
a_k\delta'\left(\LL\bigl(Z_n(k)\bigr), \LL\bigl(Z'_n(k)\bigr)\right).
\end{align*}
But $\LL\bigl(Z_n(k)\bigr)=\LL\bigl(X_n(k) \given X_\ell=x_\ell\bigr)$, and since the process ${\left(X_n(k)\right)}_{n\le0}$ is Markovian with respect to the filtration $\FF$, the latter is also equal to $\LL\bigl(X_n(k) \given X_\ell(k)=x_\ell(k)\bigr)$. 
\end{proof}

\begin{ppsition}\label{ppsition:kantomulti}
  Let ${(\rho_n)}_{n\le 0}$ be the sequence of iterated Kantorovich pseudometrics associated to the Markov process  ${(X_n)}_{n \leq 0}$, initiated by $\rho$ on $A_0$. 
  Then for any $x_n,x'_n$ in $A_n$,
 $\rho_n(x_n,x'_n)$ is the Kantorovich distance between 
  $\LL(X_0 \given X_n=x_n)$ and  $\LL(X_0 \given X_n=x'_n)$ for every 
  $n \leq -1$, and it is given by 
$$\rho_n(x_n,x_n') = \sum_{k=1}^d a_k{\delta}_n\bigl(x_n(k),x'_n(k)\bigr)$$ 
where ${\delta}_n$ is the iterated Kantorovich pseudometric associated to 
the Markov process ${\bigl(X_n(k)\bigr)}_{n \leq 0}$, initiated by $\delta$. 
\end{ppsition}

\begin{proof}
By Lemma~\ref{lemma:linear-Kanto}, the Kantorovich pseudometrics $\delta'$ in Lemma~\ref{lemma:kanto_mordered} are linear. 
Therefore, we can iteratively use Lemma~\ref{lemma:kanto_mordered} to get  the $n$-th iterated Kantorovich pseudometrics: For any $x_n,x'_n$ in $A_n$,
$$\rho_n(x_n,x_n') = \sum_{k=1}^d a_k{\delta}_n\bigl(x_n(k),x'_n(k)\bigr).$$

By Lemma~\ref{lemma:fullymonotonic}, each process ${\left(X_n(k)\right)}_{n\le0}$ is monotonic. Thus we can apply
Proposition \ref{ppsition:KantoMarkov} (the unidimensional case), which gives that ${\delta}_n(x_n(k),x_n'(k))$ is the 
Kantorovich pseudometric between $\LL\bigl(X_0(k) \given X_n(k)=x_n(k)\bigr)$ and 
 $\LL\bigl(X_0(k) \given X_n(k)=x'_n(k)\bigr)$. Then 
  $\rho_n(x_n,x_n')$ 
 is   the Kantorovich distance between 
  $\LL(X_0 \given X_n=x_n)$ and  $\LL(X_0 \given X_n=x'_n)$ 
  by  Lemma~\ref{lemma:kanto_mordered}. 
\end{proof}

\begin{thm}\label{thm:fullymonotonicmulti} 
Let ${(X_n)}_{n \leq 0}$ be an $\RR^d$-valued strongly monotonic Markov process. 
If it is identifiable, then the equivalent conditions of Theorem \ref{thm:monotonicmulti} 
are also equivalent to $\LL(X_0 \given \FF_{-\infty})=\LL(X_0)$.  
\end{thm}

\begin{proof}
This is a consequence of Proposition \ref{ppsition:kantomulti}, 
Lemma \ref{lemma:VershikMarkov}, and Lemma \ref{lemma:degenerate}. 
\end{proof}

\section{Standardness of adic filtrations}\label{sec:adic}

Standardness of adic filtrations associated to Bratteli graphs 
has become an important topic since the recent discoveries of 
Vershik~\cite{Ver13,VerIntrinsic}. 
As we will explain in Section~\ref{sec:centralwalks}, 
these are the filtrations induced by ergodic central measures on the path 
space of a Bratteli graph.  

We will apply Theorem~\ref{thm:monotonic} to derive standardness of some well-known examples 
of adic filtrations, namely those corresponding to the Pascal and the Euler graphs (Sections~\ref{sec:Pascal} and~\ref{sec:Euler}).

Actually, as we will see, it is straightforward from our Theorem~\ref{thm:monotonic} that 
every ergodic central probability measure on the one-dimensional Pascal  graph 
induces  
a standard filtration (by {\it (e) $\implies$ (a)}). 
But Theorem~\ref{thm:monotonic} is also practical to check 
the ergodicity of the random walk (using {\it(d)} or {\it 2)}). 
For the Euler graph we cannot directly apply  Theorem~\ref{thm:monotonic} 
because of multiple edges. 
Lemma~\ref{lemma:multipleedges} will allow us to deal with this situation. 

In Section~\ref{sec:mpascal} we will  apply Theorem~\ref{thm:monotonicmulti} to 
get standardness of the adic filtrations corresponding to 
the multidimensional Pascal graph. 

\subsection{Adic filtrations and other filtrations on Bratteli graphs}\label{sec:centralwalks}

 Some examples of Bratteli graphs are shown in Figure~\ref{fig:bratelli}. 
 Usually Bratteli graphs are graded by the nonnegative integers $\N$ but for our 
 purpose it is more convenient to consider the nonpositive integers $-\N$ as 
 the index set of the levels of the graphs. 
 Thus, the set of vertices $\Vb$ and the set of edges $\Eb$ of a 
 Bratteli graph $B = (\Vb,\Eb)$ have the form 
$\Vb= \cup_{n \leq 0} \Vb_n$ and  $\Eb= \cup_{n \leq 0} \Eb_n$ 
where $\Vb_n$ denotes the set of vertices at level $n$ and $\Eb_n$ denotes 
the set of edges connecting  levels $n-1$ and $n$. 
The $0$-th level set of vertices $\Vb_0=\{v_0\}$ actually consists of a single 
vertex $v_0$. Each vertex of level $n$ is assumed to be connected to at least one vertex at level $n-1$ and, if $n<-1$, to at least one vertex at level $n+1$. 

 \begin{figure}[htp]
   \centering
   \begin{subfigure}{0.43\textwidth}
   \centering
   	\includegraphics[scale=0.5]{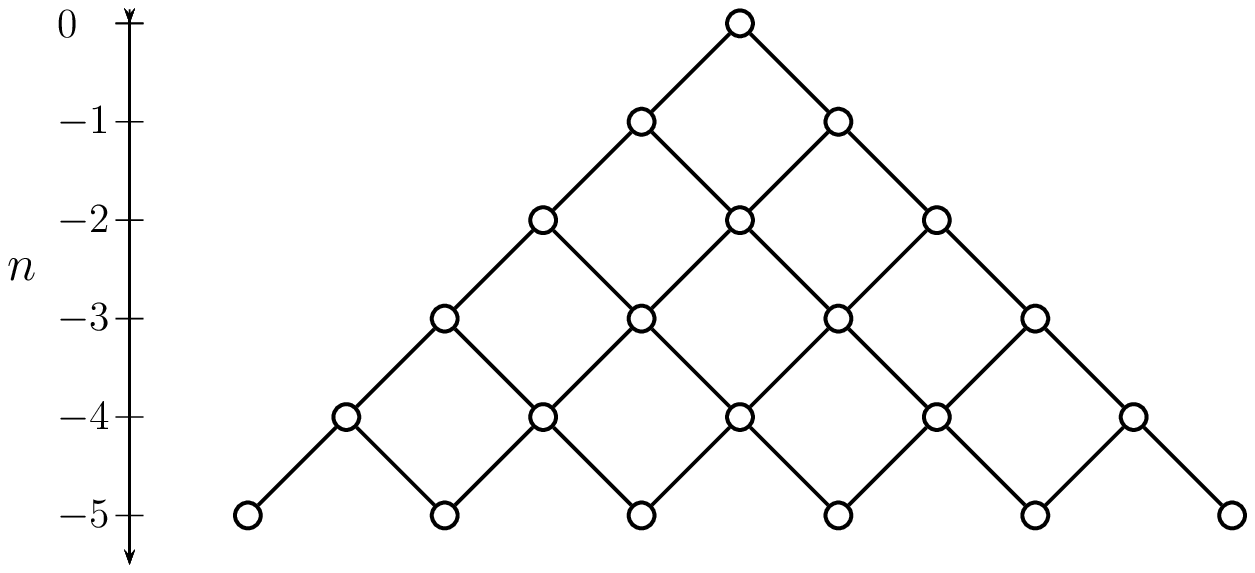}
 		\caption{Pascal}\label{fig:pascal}
    \end{subfigure}              
   \qquad
    \begin{subfigure}{0.43\textwidth}
    \centering
   	\includegraphics[scale=0.5]{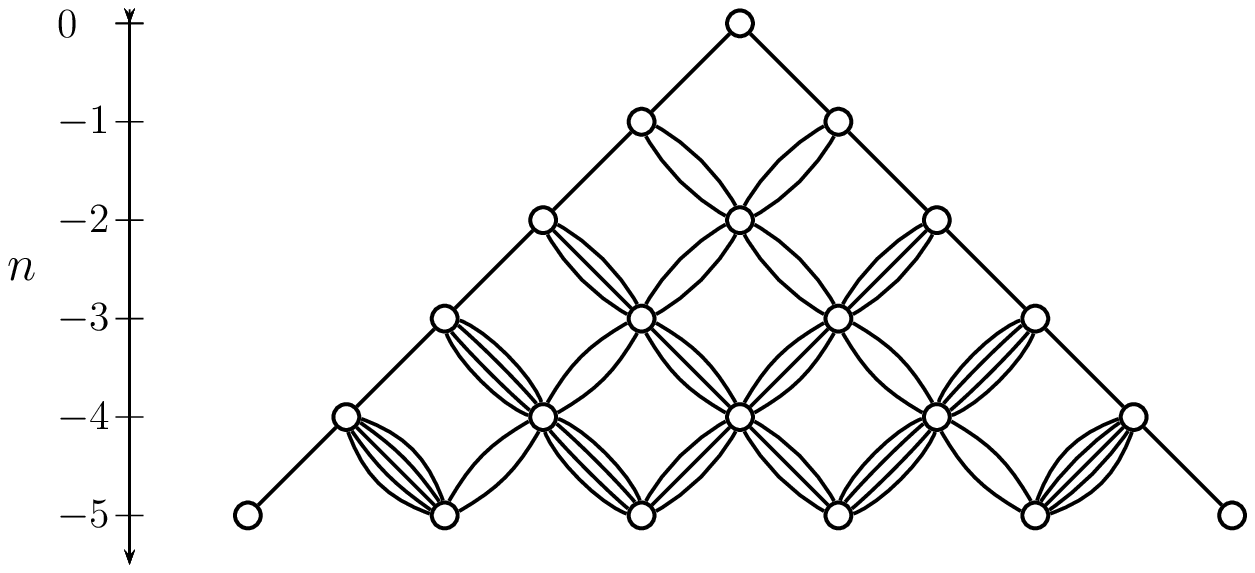}
 		\caption{Euler}\label{fig:euler}
 	\end{subfigure}
 
    \begin{subfigure}{0.43\textwidth}
    \centering
   	\includegraphics[scale=0.45]{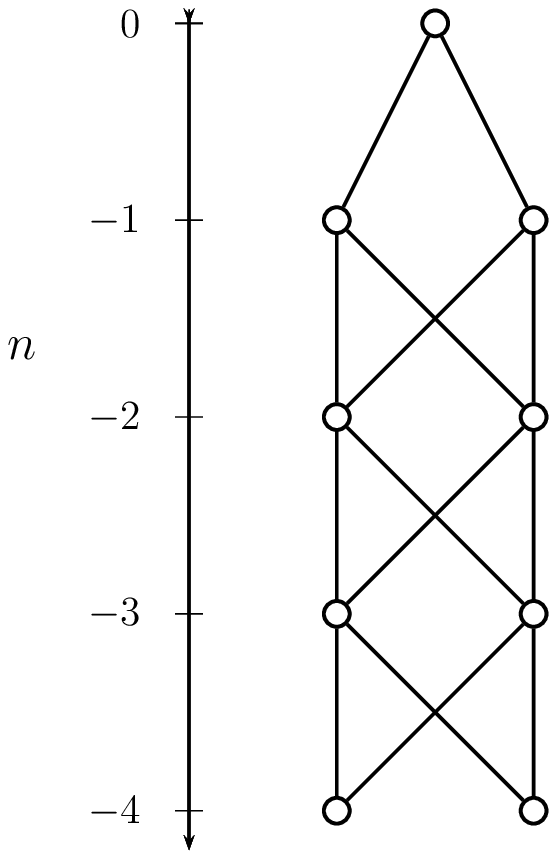}
 		\caption{Odometer}\label{fig:odometer}
   \end{subfigure}       
   \qquad
     \begin{subfigure}{0.43\textwidth}
     \centering
 		\includegraphics[scale=0.45]{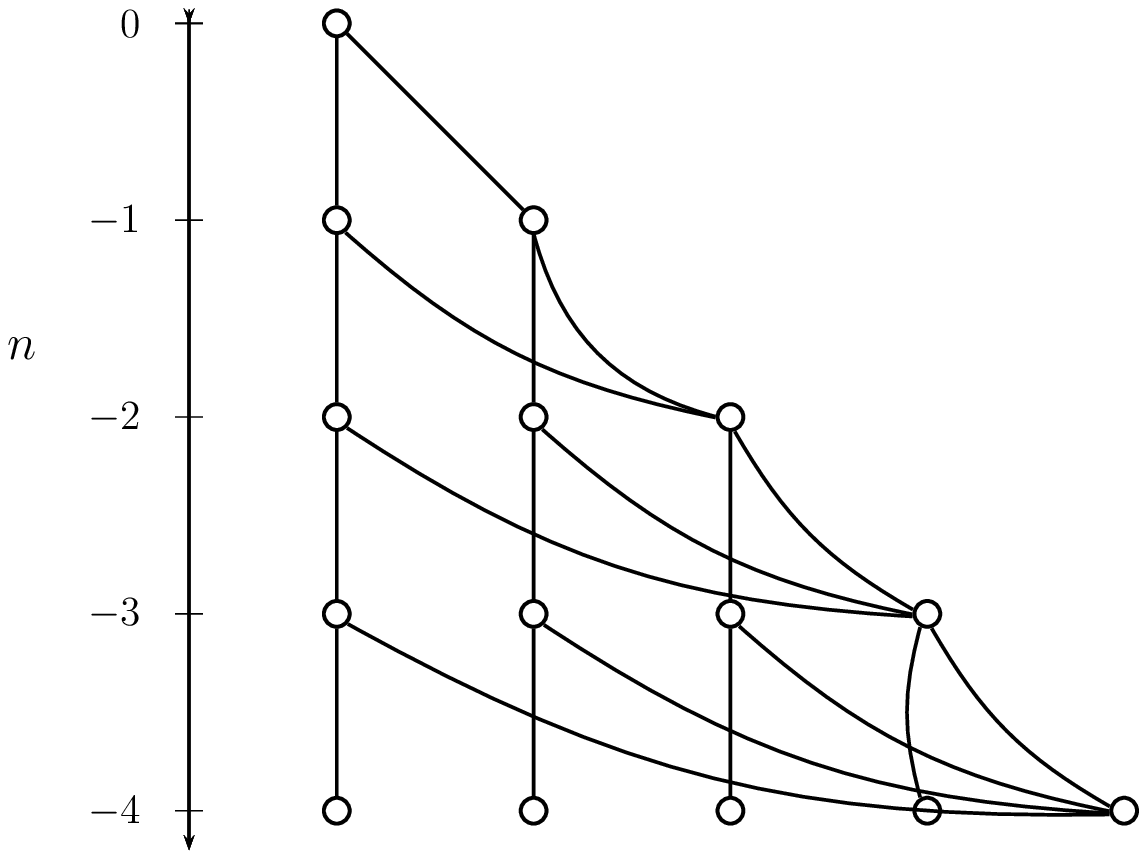}
 		\caption{Next-jump graph}\label{fig:bratelli_NJ}
    \end{subfigure}
   \caption{Four Bratteli graphs}
   \label{fig:bratelli}
 \end{figure}

There can also exist multiple edges connecting two vertices (see Euler graph).
For every vertex $v \in \Vb_n$, $n<0$, we put  \emph{labels} on the 
set of edges connecting $v$ to level $n+1$ (see Figure~\ref{fig:Euler_step}). 

 \begin{figure}[!h]
 \centering
 \scalebox{0.7}{
 	\includegraphics{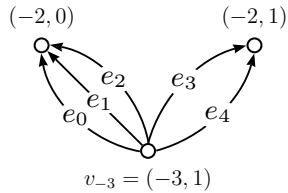}
 }
 \caption{Labeling edges in the Euler graph}
 \label{fig:Euler_step}
 \end{figure}

We denote by $\Gamma_B$ the set of infinite paths, where, 
as usual, an infinite path is a sequence 
$\gamma = {(\gamma_n)}_{n \leq 0} \in \prod_{n \leq 0} \Eb_n$ of connected edges 
starting at $v_0$,   and passing through exactly one vertex at each level $n \leq -1$.  
The path space $\Gamma_B$ has a natural Borel structure and 
any probability $\mu$ on $\Gamma_B$ can be interpreted as the law of 
a random path ${(G_n)}_{n \leq 0}$.  
The filtration $\GG$ generated by ${(G_n)}_{n \leq 0}$  is also
the filtration generated by the stochastic process ${(V_n, \eps_n)}_{n \leq 0}$   
where $V_n$ is the vertex at level $n$ of the 
random path ${(G_n)}_{n \leq 0}$ and $\eps_n$ is the 
\emph{label} of the edge connecting the vertices $V_{n-1}$ and $V_n$. 
When the graph has no multiple edges then $\GG$ is also the filtration generated by the random walk on the vertices 
 ${(V_n)}_{n \leq 0}$.   

By Rokhlin's correspondence (see \cite{Coud}), and up to measure algebra isomorphism, 
 the filtration $\GG$ corresponds to the increasing sequence of 
measurable partitions ${(\xi_n)}_{n \leq 0}$ on $(\Gamma_B,\mu)$, where 
$\xi_n$ is the measurable partition of $\Gamma_B$ into the equivalence classes 
of the equivalence relation 
${\cal R}_n$ defined by $\gamma {\cal R}_n \gamma'$ 
if $\gamma_m = \gamma'_m$ for all $m \leq n$. 
The probabilistic definition of centrality of the probability measure $\mu$, 
given below,  amounts to say that 
$\mu$ is invariant for the 
 \emph{tail equivalence relation}  
 ${\cal R}_{-\infty}$ defined by  $\gamma {\cal R}_{-\infty} \gamma'$ 
if $\gamma_m = \gamma'_m$ for  $|m|$ large enough. 

\begin{definition}\label{def:central} 
The probability measure $\mu$ on $\Gamma_B$ is \emph{central} if for each $n<0$,
the conditional distribution of $(G_{n+1}, \ldots, G_0)$ given $\GG_n$ is  
 uniform  on the set of paths connecting the vertex $V_n$ to the root of the graph. 
\end{definition}

The probabilistic property of $\GG$ corresponding to 
ergodicity of this tail equivalence relation with respect to $\mu$ is the 
degeneracy of the tail $\sigma$-field:

\begin{definition}\label{def:ergodic} 
The probability measure $\mu$ on $\Gamma_B$ is \emph{ergodic} if 
$\GG$ is Kolmogorovian. 
\end{definition}


When $\mu$ is central then the process ${(V_n, \eps_n)}_{n \leq 0}$ 
 as well as the random walk on the vertices   
 ${(V_n)}_{n \leq 0}$ are Markovian. 
  More precisely, 
  ${(V_n)}_{n \leq 0}$ is  Markovian with respect to 
 the filtration $\GG$ generated by  ${(V_n, \eps_n)}_{n \leq 0}$;   
 in other words, the filtration $\FF$ generated by ${(V_n)}_{n \leq 0}$ 
 is immersed in $\GG$. 
 Furthermore the conditional distribution of $V_{n+1}$ given  $V_n=v_n$ 
 is given by 
\begin{equation}\label{eq:central}
\Pr(V_{n+1}=v_{n+1} \given V_n=v_n) = 
m(v_n,v_{n+1})\frac{\dim(v_{n+1})}{\dim(v_n)}
\end{equation} 
where $m(v_n,v_{n+1})$ is the number of edges connecting 
$v_n$ and $v_{n+1}$, and $\dim(v)$ denotes the number of paths 
from vertex $v$ to the final vertex $v_0$. 

Centrality and ergodicity of $\mu$ also correspond to 
invariance and ergodicity of the so-called \emph{adic transformation} 
$T$ on $\Gamma_B$, and in this case the tail equivalence relation 
${\cal R}_{-\infty}$ defines the partition of $\Gamma_B$ into the 
orbits of the adic transformation. 
Standardness of $\GG$ is stronger  than ergodicity of $\mu$, 
but note that standardness of $\GG$ under a central 
ergodic measure $\mu$ is not a priori a property about the corresponding 
adic transformation, since 
the adic transformation on a Bratteli graph is possibly isomorphic to 
the adic transformation on another Bratteli graph, and these two different Bratteli 
graphs can generate non-isomorphic filtrations. 
For example the dyadic odometer is isomorphic to 
an adic transformation on the graph shown on Figure~\ref{fig:odometer} as well as 
 an adic transformation on the graph shown on Figure~\ref{fig:bratelli_NJ}.  
 The usual adic 
 representation of the dyadic odometer is given by the graph shown in 
  Figure~\ref{fig:odometer}. One easily sees that there is a unique 
  central probability  measure,  and that the corresponding Markov process  ${(V_n)}_{n \leq 0}$ 
  is actually a sequence of i.i.d.\ random variables having 
  the  uniform distribution on $\{0,1\}$.
 Therefore $\GG$ is obviously a standard filtration. 
 The Bratteli graph of Figure~\ref{fig:bratelli_NJ} 
shows another possible adic representation 
of  the dyadic odometer. 
Standardness of the corresponding filtration $\GG$ has been studied 
in~\cite{LauNJ} and~\cite{LauEntropy} in the case when $\mu$ is any 
 independent product of Bernoulli measures on the path space, 
 and this includes all the central ergodic measures.  
 In Sections~\ref{sec:Pascal} and~\ref{sec:Euler} 
 we will use Theorem~\ref{thm:monotonic} to study the case of the Pascal graph 
 (Figure~\ref{fig:pascal}) and the case of the Euler graph (Figure~\ref{fig:euler}). 

\medskip 
The lemma below is useful to establish standardness in 
 the case of a graph with multiple edges, 
such as the Euler graph. 
 Note that the conditional independence assumption 
 $\LL(\eps_n \given V_{n-1})= \LL(\eps_n \given \GG_{n-1})$ 
 of this lemma implies that ${(V_n)}_{n \leq 0}$ is Markovian, and 
 this assumption is always fulfilled for a central measure. 
 
\begin{lemma}\label{lemma:multipleedges}
Let $\GG$ be the filtration associated to a probability measure on 
the path space of a Bratteli graph,  
and denote by ${(V_n, \eps_n)}_{n \leq 0}$ the stochastic process generating $\GG$,  
where $V_n$ is the vertex at level $n$ and $\eps_n$ is the label of the edge 
connecting $V_{n-1}$ to $V_n$. 
Assume that $\LL(\eps_n \given V_{n-1})= \LL(\eps_n \given \GG_{n-1})$, 
that is to say $\eps_n$ is conditionally independent of $\GG_{n-1}$ 
given $V_{n-1}$. 
Denote by $\FF$ the filtration of the random walk ${(V_n)}_{n \leq 0}$ on the 
vertices. 

Then
\begin{enumerate}[1)]
\item there exists a parameterization ${(U_n)}_{n \leq 0}$   of 
$\FF$ which is also a parameterization of $\GG$, and 
such that the parametric extension 
of $\FF$ with   ${(U_n)}_{n \leq 0}$ (Definition~\ref{def:extension_superinnovation}) 
is also  the parametric extension 
of $\GG$ with   ${(U_n)}_{n \leq 0}$; 

\item assuming $\Vb_n \subset \RR$ and ${(V_n)}_{n \leq 0}$  monotonic, 
there exists a monotonic parametric representation  ${(f_n, U_n)}_{n \leq 0}$  of 
${(V_n)}_{n \leq 0}$ 
with a parameterization ${(U_n)}_{n \leq 0}$  satisfying the above properties.
\end{enumerate}
\end{lemma}

\begin{proof}
Assume without loss of generality that the labels of the edges are real numbers. 
Denote by $\phi_n$ a measurable function such that $V_n = \phi_n(V_{n-1}, \eps_n)$, 
and denote by  $h_n(v_{n-1}, \cdot)$ the 
right-continuous inverse of the cumulative distribution function of the conditional law 
$\LL(\eps_n \given V_{n-1}=v_n)$.  
Then the function $f_n$ defined by 
$$f_n(v_{n-1}, \cdot) = \phi_n\bigl(v_{n-1},h_n(v_{n-1}, \cdot)\bigr)$$ is 
an updating function of the Markov kernel $\PP(V_n \in \cdot \given V_{n-1}=v_{n-1})$. 

Consider a copy ${(V'_n)}_{n \leq 0}$ of the process ${(V_n)}_{n \leq 0}$ 
given by a parametric representation ${(f_n, U'_n)}_{n \leq 0}$ 
 with these updating functions $f_n$, and 
set $\eps'_n = h_n(V'_{n-1}, U'_n)$. 
Then it is not difficult to see that the 
process ${(V'_n,\eps'_n)}_{n \leq 0}$ is a copy of ${(V_n,\eps_n)}_{n \leq 0}$. 
Moreover, denoting by $\GG'$ its filtration, $U'_n$ 
is independent of $\GG'_{n-1}$, and  
$$\GG'_n = \GG'_{n-1} \vee \sigma(\eps'_n) \subset \GG'_{n-1} \vee \sigma(U'_n),$$ 
thereby showing that ${(U'_n)}_{n \leq 0}$ 
is  a parameterization of $\GG'$. 
This proves 1). 

Assuming now $\Vb_n \subset \RR$, it is always possible to take right-continuous increasing  functions 
$\phi_n(v_{n-1}, \cdot)$. 
With such a choice, the function $f_n$ constructed above is the quantile updating function 
\eqref{eq:quantile}, and then the representation is monotonic 
whenever ${(V_n)}_{n \leq 0}$ is monotonic. 
\end{proof}

We cannot deduce from result 1) of Lemma~\ref{lemma:multipleedges} that $\GG$ admits 
a generating parameterization whenever $\FF$ admits a generating parameterization. 
But thanks to this result and to Proposition~6.1 in~\cite{LauTeoriya}, which says that 
standardness is hereditary under parametric 
extension, we know that 
\emph{$\FF$ is standard if and only if $\GG$ is standard}.  
This result is not used in the present paper but it is useful for 
the study of other Bratteli graphs.  
  
\subsection{Vershik's intrinsic metrics}\label{sec:intrinsic}

Given a probability measure $\mu$ on $\Gamma_B$, for which the process ${(V_n)}_{n\le0}$ is Markovian, 
we can consider the iterated Kantorovich pseudometrics $\rho_n$ defined as in Section~\ref{sec:vershik}.
 But since $\Vb_0$  is always reduced to a singleton, we start from a metric $\rho_{-1}$ 
 defined on the set $\Vb_{-1}$ instead of a metric $\rho_0$ on $\Vb_0$.  
 Each pseudometric $\rho_n$, $n\le -1$ is then defined on the set $\Vb_n$ of vertices of level $n$.  These pseudometrics only depend on the Markov kernels $P_n$, in particular all central probability measures will give rise to the same sequence of pseudometrics. The pseudometrics $\rho_n$ obtained in the  case of a central measure have been introduced by Vershik in~\cite{VerIntrinsic}, who called them \emph{intrinsic pseudometrics}. 
 In the next sections we will provide the intrinsic metrics $\rho_n$ 
 for the Pascal graph and the Euler graph with the help 
 of  Proposition~\ref{ppsition:KantoMarkov}, and for the higher dimensional Pascal graph 
 with the help of Proposition~\ref{ppsition:kantomulti}. 
 
 Applying the theorems of~\cite{VerIntrinsic} about the identification of 
 the ergodic central measures is  beyond the scope of this paper.  
 This is based on the intrinsic pseudometric defined on the whole set of vertices 
 $\cup_{n \leq 0} \Vb_n$ and extending all the $\rho_n$, 
 which we will not explicit here. 
 Our derivation of the $\rho_n$ provides a helpful starting point 
 for further work in this direction.

Recall that the $\rho_n$ are metrics  under  the identifiability 
of the associated Markov process ${(V_n)}_{n\le0}$ 
(Definition~\ref{def:identifiable}), and identifiability is easy to check 
in the case of central measures. 
It is equivalent to the following property: 
	\emph{For each $n<-1$, for any two different vertices $v,v'\in\Vb_n$, 
	there exists at least one vertex $w\in\Vb_{n+1}$ such that the number 
	of edges connecting $v$ and $w$ is different from  the number of edges 
	connecting $v'$ and $w$}. 
For a graph without multiple edge, this simply means that 
$v$ and $v'$ are not connected to the same set of vertices at level $n+1$. 


\section{Pascal filtration}\label{sec:Pascal}

Consider the $(-\N)$-graded Pascal graph shown in Figure~\ref{fig:Pascal1}. 
At each level $n$, we label the vertices $0$, $1$, $\ldots$, $|n|$. 
Then a vertex can be identified by the pair $(n,k)$ consisting in its level 
$n$ and its label $k$, but when the level is understood we simply use the label 
as the identifier. 
Each vertex $v$ at level $n$ is connected to vertices $v$ and $v+1$ at level $n-1$.
There is no multiple edge and a random path in the graph corresponds to a random walk ${(V_n)}_{n \leq 0}$ 
on the vertices of the graph, 
where $V_n$ is a vertex at level $n$ and $(V_{n},V_{n-1})$ are connected. 

 \begin{figure}[!h]
 \centering
     \begin{subfigure}{0.43\textwidth}
        \centering
        \includegraphics[scale=0.6]{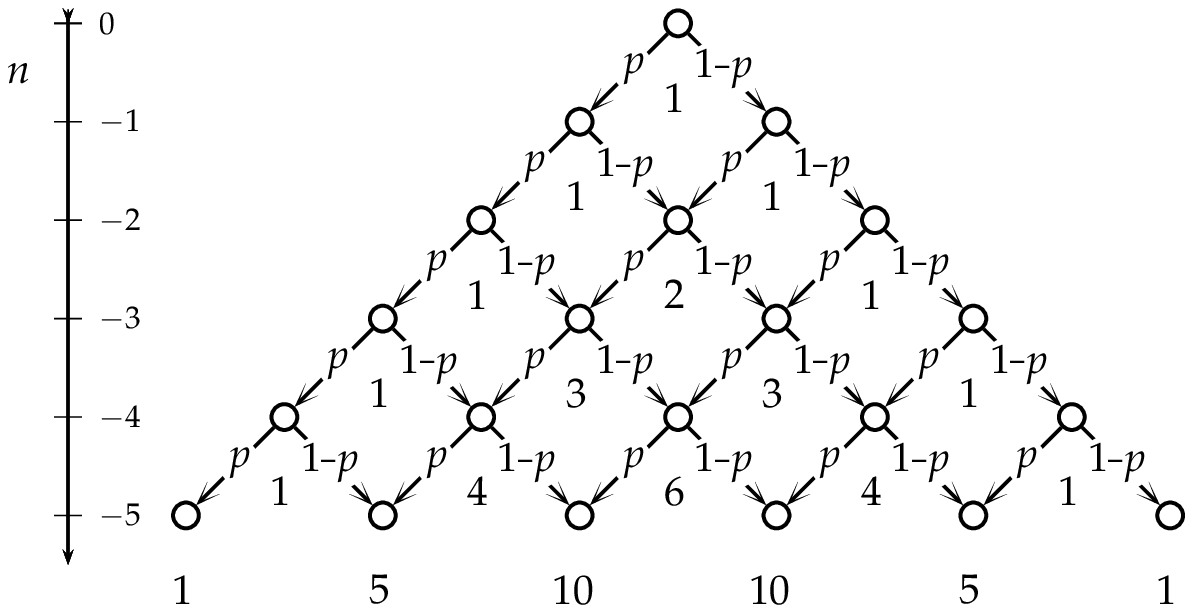}
        \caption{Pascal walk from $n=0$ to $n=-\infty$}\label{fig:Pascal1}
     \end{subfigure}          
     \qquad  
     \begin{subfigure}{0.43\textwidth}  
        \centering
        \includegraphics[scale=0.6]{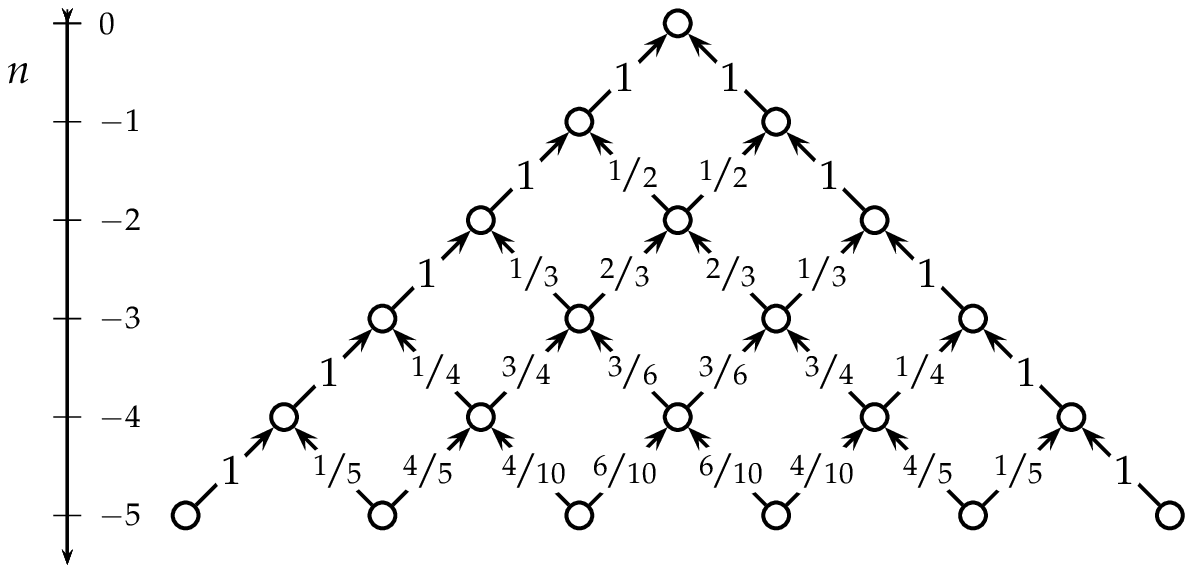}
        \caption{Pascal walk from $n=-\infty$ to $n=0$}\label{fig:Pascal2}
     \end{subfigure} 
 \caption{Pascal random walk}
 \label{fig:Pascal}
 \end{figure}

The path space of the Pascal graph is naturally identified with ${\{0,1\}}^{-\N}$. 
Under any central probability measure, the process ${(V_n)}_{n \leq 0}$ 
obviously is a monotonic and identifiable Markov process 
(definitions~\ref{def:monotonic} and~\ref{def:identifiable}). 
Its Markovian transition distributions ${\cal L}(V_n \given V_{n-1}=v)$ 
are easy to derive with the help of formula \eqref{eq:central}. 
They are shown 
in Figure~\ref{fig:Pascal2} for $n=0$ to $n=-4$. 
The only thing we will need is the conditional law 
${\cal L}(V_{-1} \given V_{n}=v_n)$ and it is not difficult to see that 
it is the distribution on $\{0,1\}$ given by 
$\PP(V_{-1}=1 \given V_n=v_n)=\frac{v_n}{|n|}$. 

\subsection{Standardness}\label{sec:Pascal_standard}

It has been shown (see e.g.~\cite{MP}) that the ergodic central probability measures are 
those for which the reverse random walk $(V_0, V_{-1}, \ldots)$ is Markovian 
with a constant Markovian transition $(p, 1-p)$ as shown in Figure~\ref{fig:Pascal1}. 
In other words the ergodic central probability measures are the infinite product 
Bernoulli measures $(p, 1-p)$. Then $V_n$ has the binomial distribution $\Bin(|n|,p)$.

Using Theorem~\ref{thm:monotonic}, we can directly show standardness of the filtration
$\FF$ generated by $(V_n)$ under these  infinite product 
Bernoulli measures.

\begin{ppsition}\label{ppsition:pascalstandard}
When $\mu$ is an  infinite product 
Bernoulli measure $(p, 1-p)$ then the random walk ${(V_n)}_{n \leq 0}$ is a 
monotonic Markov process generating a standard filtration.  In particular, this measure is ergodic.
\end{ppsition}

\begin{proof}
Obviously, ${(V_n)}_{n \leq -1}$ is a 
monotonic and identifiable Markov process (see last paragraph in Section~\ref{sec:intrinsic}). 
We check criterion {\it 2)} in  Theorem~\ref{thm:monotonic}.  
The conditional distribution $\mu_{v_n}:=\LL(V_{-1} \given V_n=v_n)$ 
is the law on $\{0,1\}$ given by 
$\mu_{v_n}(1)= \frac{v_n}{|n|}$, thus the conditional law $\LL(V_{-1} \given \FF_n)$ 
 goes to $\LL(V_{-1})$ by the law of large numbers and then  Theorem~\ref{thm:monotonic} applies 
 in view of Lemma~\ref{lemma:convergence_conditional_law}. 
\end{proof}

In fact, as long as the process ${(V_n)}_{n \leq 0}$ is a Markov process for some 
 probability measure on $\Gamma_B$, 
 it is easy to see that it is necessarily a monotonic Markov process. 
 We then get the following consequence of Theorem~\ref{thm:monotonic} (by {\it (e) $\implies$ (a)}).

\begin{thm}
  \label{thm:all_Markov_processes_in_Pascal}
  For any ergodic probability measure on $\Gamma_B$ under which ${(V_n)}_{n \leq 0}$  is a Markov process, the filtration $\FF$ generated by 
  ${(V_n)}_{n \leq 0}$  admits a generating parameterization, hence is standard.
\end{thm}

\subsection{Intrinsic metrics on the Pascal graph}\label{sec:Pascal_intrinsic}

We did not need to resort to Vershik's standardness criterion 
(Lemma~\ref{lemma:VershikMarkov}) to 
prove standardness of the Pascal adic filtrations  
(Proposition~\ref{ppsition:pascalstandard}). However, 
as we mentioned in Section~\ref{sec:intrinsic}, it 
is interesting to have a look at the intrinsic metrics $\rho_n$ on the state 
space $\Vb_n=\{0, \ldots, |n|\}$ of $V_n$, starting from the $0$-$1$ distance on 
$\Vb_{-1}$. 
The $\rho_n$ are easily obtained by Proposition~\ref{ppsition:KantoMarkov}: 
the distance $\rho_n(v_n, v'_n)$ is nothing but the Kantorovich distance between 
$\LL(V_{-1} \given V_n=v_n)$ and $\LL(V_{-1} \given V_n=v'_n)$, and then 
$$
\rho_n(v_n, v'_n) = \frac{|v_n-v'_n|}{|n|},
$$
wherefrom it is not difficult to apply Lemma~\ref{lemma:VershikMarkov} to get 
standardness of $\FF$. 
The space $(\Vb_n,\rho_n)$ is isometric to the subset 
$\bigl\{\frac{k}{|n|}, k = 0, \ldots, |n|\bigr\}$ 
of the unit interval $[0,1]$. 
Figure~\ref{fig:PascalKanto} shows an embedding of the Pascal graph 
in the plane such that $\rho_n$ is given by the Euclidean distance at each level $n$. 

\begin{figure}[!h]
\centering
       \includegraphics[scale=0.7]{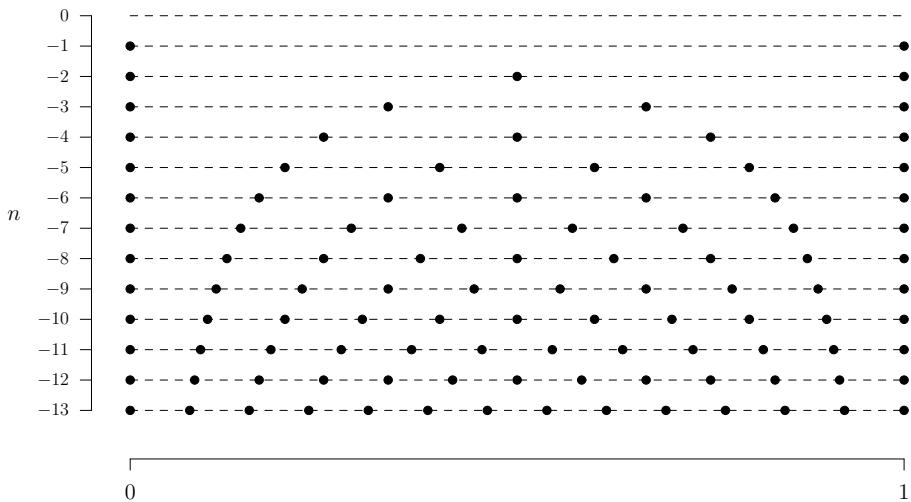}
\caption{The Pascal graph under the intrinsic metrics}
\label{fig:PascalKanto}
\end{figure}


\section{Euler filtration}\label{sec:Euler}

The Euler graph, shown on Figure~\ref{fig:Euler1} from level 
$n=0$ to level $n=-5$, 
has the same vertex set as the Pascal graph, but has multiple edges: Vertex $v$ of level $n$ is connected to vertex $v$ of level $n-1$ by $v+1$ edges, and to vertex $v+1$ of level $n-1$ by $|n|+1-v$ edges. We refer to~\cite{FP,FKPS,PV} for properties of this graph. In particular, the number of paths connecting vertex $v$ of level $n \leq -1$ to 
the root vertex at level $0$ is the \emph{Eulerian number} $A(|n|+1, v)$. 


 \begin{figure}[!h]
 \centering
     \begin{subfigure}{\textwidth}
        \centering
        \includegraphics[scale=0.8]{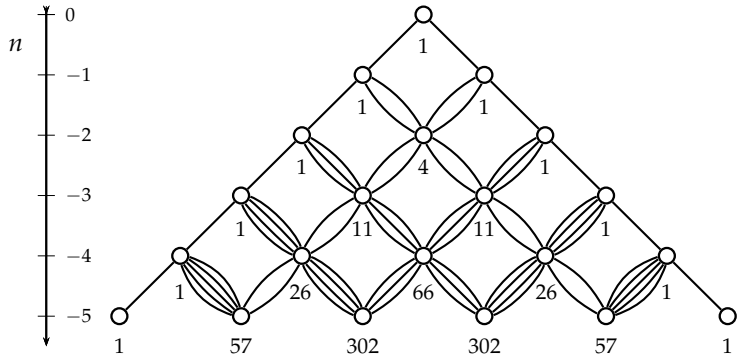}
        \caption{Euler graph}\label{fig:Euler1}
     \end{subfigure}          
     
     
     \vspace{5mm}
     
     \begin{subfigure}{\textwidth}  
        \centering
        \includegraphics[scale=0.8]{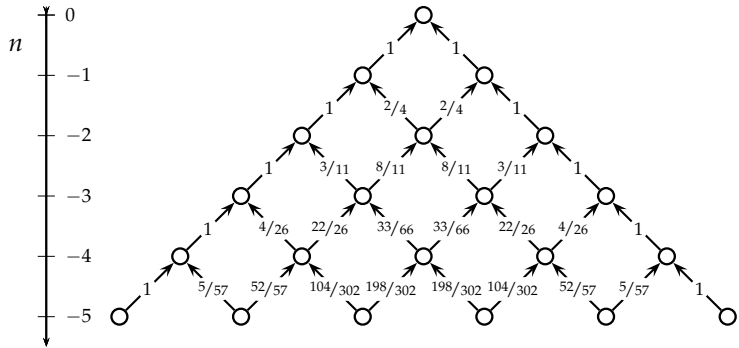}
        \caption{Walk on the vertices}\label{fig:Euler2}
     \end{subfigure} 
 \caption{Euler random walk}
 \label{fig:Euler}
 \end{figure}
 
 It is shown in~\cite{GO} that there exist countably many ergodic central measures on $\Gamma_B$ for this graph. However, only one of them, called the \emph{symmetric measure}, has full support, 
 as shown in~\cite{FP} (the others are concentrated on paths whose distance to one of the sides of the triangle is bounded).

Given a probability measure on $\Gamma_B$,   
  as explained in Section~\ref{sec:adic}, 
we consider a stochastic process ${(G_n)}_{n \leq 0}$ distributed on $\Gamma_B$ according to $\mu$, 
where $G_n$ is the edge at level $n$, and we are interested in the filtration 
$\GG$ it generates. This filtration is also generated by  the 
process  ${(V_n,\eps_n)}_{n \leq 0}$, where $V_n$ is the vertex at level $n$ and 
$\eps_n$ the label connecting $V_{n-1}$ to $V_n$. 
 Under the symmetric measure, the process ${(G_n)}_{n\le0}$ is Markovian 
 and the conditional distribution of $G_{n-1}$ given $G_n$ is the uniform distribution among the $|n|+2$ edges in $\Eb_n$ connected to $G_n$. 
We will derive standardness of the filtration $\GG$ under the symmetric measure. 
 The explicit conditional distributions $\LL(V_{-1} \given V_n=v_n)$ 
 can be derived from Equation~(1.1) in~\cite{PV}, but to  show standardness 
 we will only 
use the following result coming from Equation~(1.3) in~\cite{PV}: 
$$\lim_{n\to-\infty}\PP(V_{-1}=1 \given V_n=v_n) = \frac{1}{2}$$ 
for every sequence ${(v_n)}_{n \leq 0}$  of vertices $v_n\in\Vb_n$ 
 such that both $v_n$ and $|n|-v_n$ go to infinity as $n\to-\infty$.

\subsection{Standardness}\label{sec:Euler_standard} 

For the Euler filtration we have to deal with multiple edges: 
$\GG$ is generated by the Markov process ${(V_n,\eps_n)}_{n \leq 0}$ 
 (Section~\ref{sec:adic}) and Theorem~\ref{thm:monotonic} can only provide 
 a generating parameterization of the smaller filtration $\FF$ 
 generated by  
 the random walk on the vertices ${(V_n)}_{n \leq 0}$. 
 A generating parameterization of $\GG$ will be derived by 
 applying Theorem~\ref{thm:monotonic} to ${(V_n)}_{n \leq 0}$ 
 and then by applying  Lemma~\ref{lemma:multipleedges}. 

%

\begin{lemma}
  \label{lemma:Euler}
  Under the symmetric central measure $\mu$, we have
  $$ V_n \tend[a.e.]{n}{-\infty} \infty, \quad\text{and }|n|-V_n \tend[a.e.]{n}{-\infty} \infty.$$
\end{lemma}
\begin{proof}
  Consider the Markov process $(\tilde V_n)_{n\le 0}$ where $\tilde V_n$ takes its values in $\Vb_n$, defined by the conditional distribution 
  $$ \PP(\tilde V_{n-1}=v\given \tilde V_n=v)=\PP(\tilde V_{n-1}=v+1\given \tilde V_n=v)=\frac{1}{2}.$$ 
  The process $(\tilde V_0, \tilde V_{-1}, \ldots)$ is nothing but the well-known simple 
  symmetric random walk.   
  By the law of large numbers, the property claimed for $V_n$ obviously holds for $\tilde V_n$. Moreover, we can easily construct a coupling of the two Markov processes for which, for all $n\le 0$,
  $$ \left| V_n - \frac{|n|}{2} \right| \le \left| \tilde V_n - \frac{|n|}{2} \right| \quad\text{a.s.} $$
Consequently $(V_n)$ inherits of the same property.
\end{proof}

\begin{ppsition}\label{ppsition:eulerstandard}
For the symmetric central measure $\mu$, the Euler filtration $\GG$ admits a generating parameterization, hence is standard. In particular, $\mu$ is ergodic. 
\end{ppsition}

\begin{proof}
We first check criterion {\it 2)} in  Theorem~\ref{thm:monotonic} for ${(V_n)}_{n \leq -1}$ 
which obviously is a monotonic and identifiable Markov process. 
As we previously mentioned, 
 it follows from Equation~(1.3) in~\cite{PV} 
that $\mu_{v_n}(1) \to \frac12$ whenever $(v_n)$ is a sequence of vertices such that $v_n\in\Vb_n$ 
and both $v_n$ and $|n|-v_n$ go to infinity as $n\to-\infty$. 
We recognize the distribution of $V_{-1}$ under $\mu$, and using Lemma~\ref{lemma:Euler} we see that criterion {\it 2)} in  Theorem~\ref{thm:monotonic} is fulfilled.
Now, by  {\it(c)} in Theorem~\ref{thm:monotonic} and  
2) in Lemma~\ref{lemma:multipleedges}, $\FF$ and $\GG$ admit a common generating parameterization. It follows by Lemma~\ref{lemma:generatingparam} that $\GG$ is standard.
\end{proof}

Similarly to Theorem~\ref{thm:all_Markov_processes_in_Pascal} about the Pascal graph, one has 
the following theorem for the Euler graph. 
\begin{thm}
  \label{thm:all_Markov_processes_in_Euler}
  Under an ergodic probability measure on $\Gamma_B$ and under 
the conditional independence assumption 
  $\LL(\eps_n \given V_{n-1})= \LL(\eps_n \given \GG_{n-1})$,  
   the filtration $\GG$ admits a generating parameterization, hence is standard. 
\end{thm}

\begin{proof}
Under the conditional independence assumption, the process 
 ${(V_n)}_{n \leq 0}$  is Markovian, and the filration $\FF$ it generates 
 admits a generating parameterization by  Theorem~\ref{thm:all_Markov_processes_in_Pascal}. 
 We conclude similarly to the proof of Proposition~\ref{ppsition:eulerstandard}, 
 combining Theorem~\ref{thm:monotonic} and   Lemma~\ref{lemma:multipleedges}. 
\end{proof}


\subsection{Intrinsic metrics on the Euler graph}\label{sec:Euler_intrinsic}

Similarly to the Pascal case, 
the intrinsic metrics $\rho_n$ on the state 
space $\Vb_n=\{0, \ldots, |n|\}$ of $V_n$, starting from the discrete distance on 
$\Vb_{-1}$, 
are easily obtained by Proposition~\ref{ppsition:KantoMarkov}: 
The distance $\rho_n(v_n, v'_n)$ is nothing but the Kantorovich distance between 
$\LL(V_{-1} \given V_n=v_n)$ and $\LL(V_{-1} \given V_n=v'_n)$. 
We can explicit these conditional laws using the formula provided by Equation~(1.1) 
in~\cite{PV}, 
which gives the number of paths 
 connecting a vertex $v_n$ at some level $n \leq -2$ to the right vertex 
 at level $-1$. The number of such paths is the generalized Eulerian number 
$$
A_{0,1}(|n|-v_n,v_n-1) = 
\sum_{t=0}^{|n|-v_n} {(-1)}^{|n|-v_n-t} {t+2 \choose t}{|n|+2 \choose |n|-v_n-t}{(1+t)}^{|n|-1}.
$$
Recalling that the total number of paths connecting vertex $v_n$ of level $n$ to the root of the graph is the classical Eulerian number $A(|n|+1, v_n)$, we get  the conditional law $\LL(V_{-1} \given V_n=v_n)$ under the centrality assumption: It is the probability on $\{0,1\}$ given by 
$$ \PP(V_{-1}=1 \given V_n=v_n) 
= \frac{A_{0,1}(|n|-v_n,v_n-1)}{A(|n|+1, v_n)}. $$
From this, we can derive the following formula giving the intrincic metric at level $n$:
$$
\rho_n(v_n, v'_n) = \left|\frac{A_{0,1}(|n|-v_n,v_n-1)}{A(|n|+1, v_n)}-\frac{A_{0,1}(|n|-v'_n,v'_n-1)}{A(|n|+1, v'_n)}\right|. 
$$
We also know by Proposition~\ref{ppsition:KantoMarkov} that the 
 space $(\Vb_n,\rho_n)$ is isometric a subset 
of the unit interval $[0,1]$. 
Figure~\ref{fig:EulerKanto} shows an embedding of the Euler graph 
in the plane such that $\rho_n$ is given by the Euclidean distance at each level $n$. 

\begin{figure}[!h]
\centering
       \includegraphics[scale=0.7]{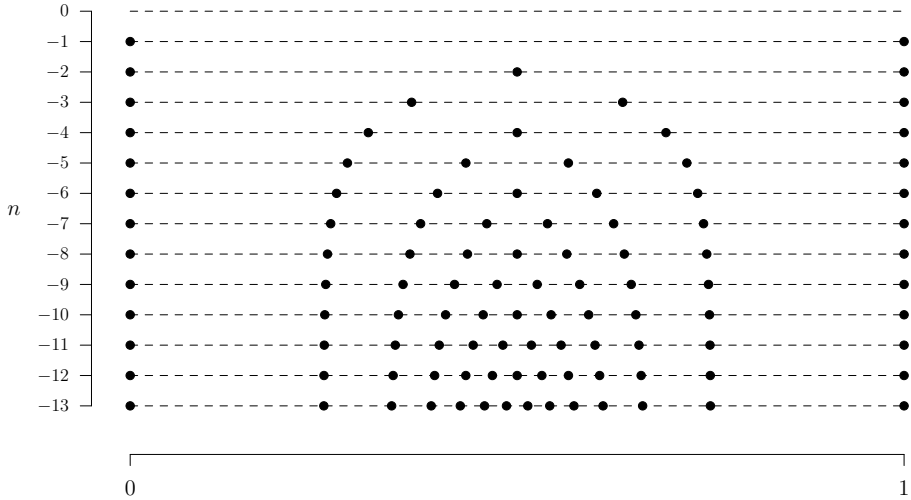}
\caption{The Euler graph under the intrinsic metrics}
\label{fig:EulerKanto}
\end{figure}

\section{Multidimensional Pascal filtration} 
\label{sec:mpascal}
Now we introduce the $d$-dimensional Pascal graph. 
The Pascal graph of Section~\ref{sec:Pascal} corresponds to the case $d=2$.
We will provide three different proofs that the filtration is standard for any dimension $d\geq 2$ 
under the known ergodic central measures.  
The first proof is an application of Theorem~\ref{thm:monotonicmulti}. 
The second proof is an application of  Theorem~\ref{thm:fullymonotonicmulti}, 
 using Proposition~\ref{ppsition:kantomulti} to  derive the intrinsic metrics $\rho_n$. 
These two proofs  only provides standardness, not a generating parameterization. 
In the third proof we construct a generating parameterization with the help of 
Theorem~\ref{thm:monotonic}.

Let $d \geq 2$ be an integer or $d=\infty$. 
Vertices of the $d$-dimensional Pascal graph are points 
$(i_1, \ldots, i_d) \in \N^d$ when $d<\infty$. 
When $d=\infty$, the vertices are the sequences 
$(i_1, i_2, \ldots) \in \N^\infty$ with finitely many nonzero terms. 
 The set of vertices at level $n$ is 
$$\Vb_n^d = \bigl\{(i_1, \ldots, i_d) \in \N^d 
 \mid i_1+\cdots + i_d=|n|\bigr\}$$ 
 and two vertices $(i_1, \ldots, i_d) \in \Vb_n^d$ and 
  $(j_1, \ldots, j_d) \in \Vb_{n-1}^d$ are connected if and only if 
  $\sum|i_k - j_k|=1$.

Since there is no multiple edge in the graph,  for any central probability measure, 
 the corresponding adic filtration $\GG$ 
is generated by the Markovian random walk on the vertices. 
Temporarily denoting by ${(V_n)}_{n \leq 0}$ this random walk, 
centrality means that the Markovian transition from $n$ to $n+1$ is given by 
\begin{equation}
  \label{eq:multi-dim-Pascal}
  \LL(V_{n+1} \given V_n=v) = \sum_{i=1}^d \frac{v(i)}{|n|}\delta_{v-e_i},
\end{equation}
where $e_i$ is the vector whose $i$-th 
term is $1$ and all the other ones are $0$. 

It is known (see~\cite{FP2}, Theorem 5.3) that a central measure is ergodic if and only if 
there is a probability vector $(\theta_1, \ldots, \theta_d)$  such that 
every Markov  transition from $n$ to $n-1$ is given by 
$$\Pr(V_{n-1}=v_n+e_i \given V_n=v_n) 
= \theta_i \quad \text{for all } i.$$ 
Under this ergodic central measure, $V_n$ has the multinomial distribution with 
parameter  $(\theta_1, \ldots, \theta_d)$ (see Figure~\ref{fig:MultiPascal}).
For this reason, 
let us term the ergodic central measures as the 
 \emph{multinomial central measures}. 
 It is not difficult to check that the multinomial central measures are ergodic, 
but in our second and third proofs of standardness we will not use ergodicity.

 \begin{figure}[th]
 \centering
 
     \begin{subfigure}{0.43\textwidth}
        \centering
        \includegraphics[scale=0.53]{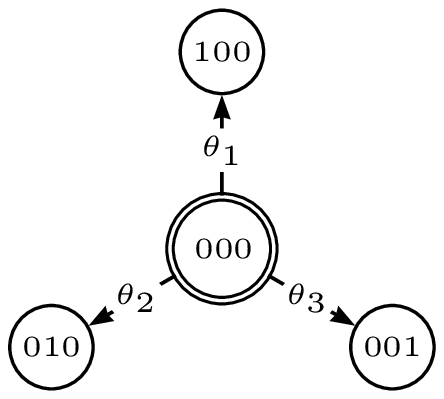}
        \caption{From $n=0$ to $n=-1$}\label{fig:a}
     \end{subfigure}          
     \qquad  
     \begin{subfigure}{0.43\textwidth}  
        \centering
        \includegraphics[scale=0.53]{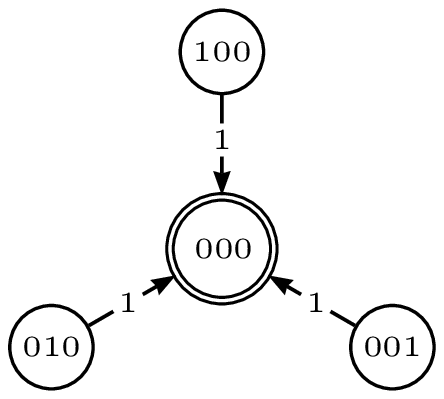}
        \caption{From $n=-1$ to $n=0$}\label{fig:aa}
     \end{subfigure}    
 
     \begin{subfigure}{0.43\textwidth}
        \centering
        \includegraphics[scale=0.53]{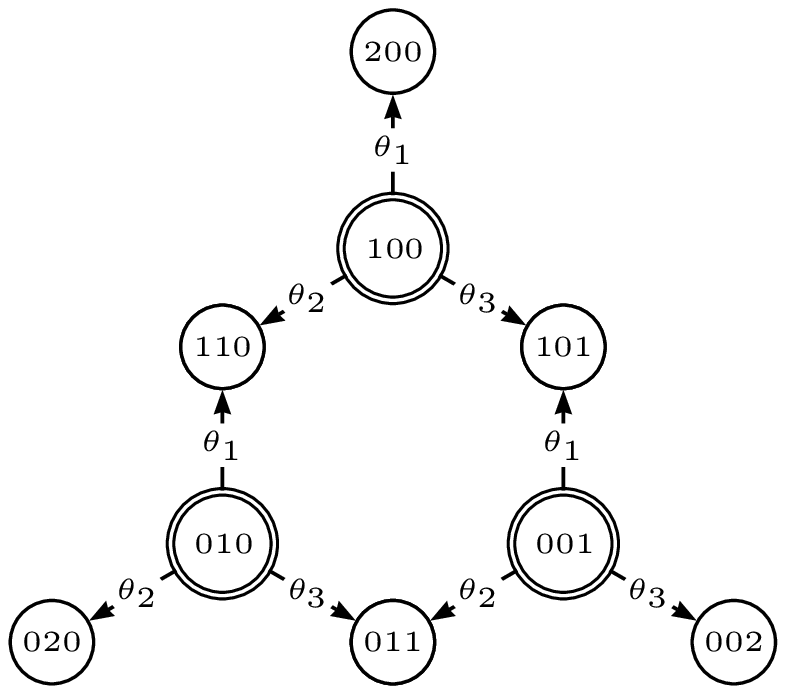}
        \caption{From $n=-1$ to $n=-2$}\label{fig:b}
     \end{subfigure}          
     \qquad 
     \begin{subfigure}{0.43\textwidth}  
        \centering
        \includegraphics[scale=0.53]{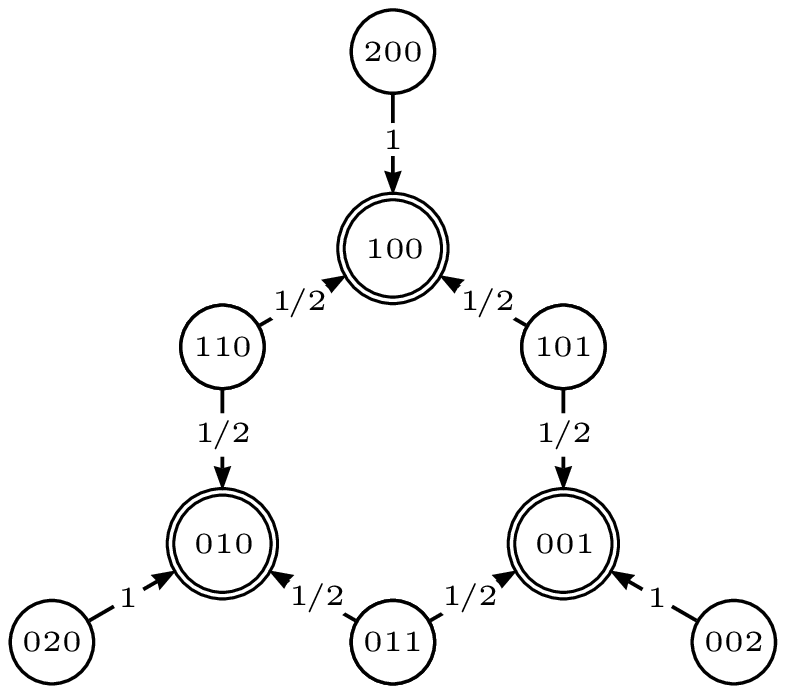}
        \caption{From $n=-2$ to $n=-1$}\label{fig:bb}
     \end{subfigure}    
 
     \begin{subfigure}{0.43\textwidth}
        \centering
        \includegraphics[scale=0.53]{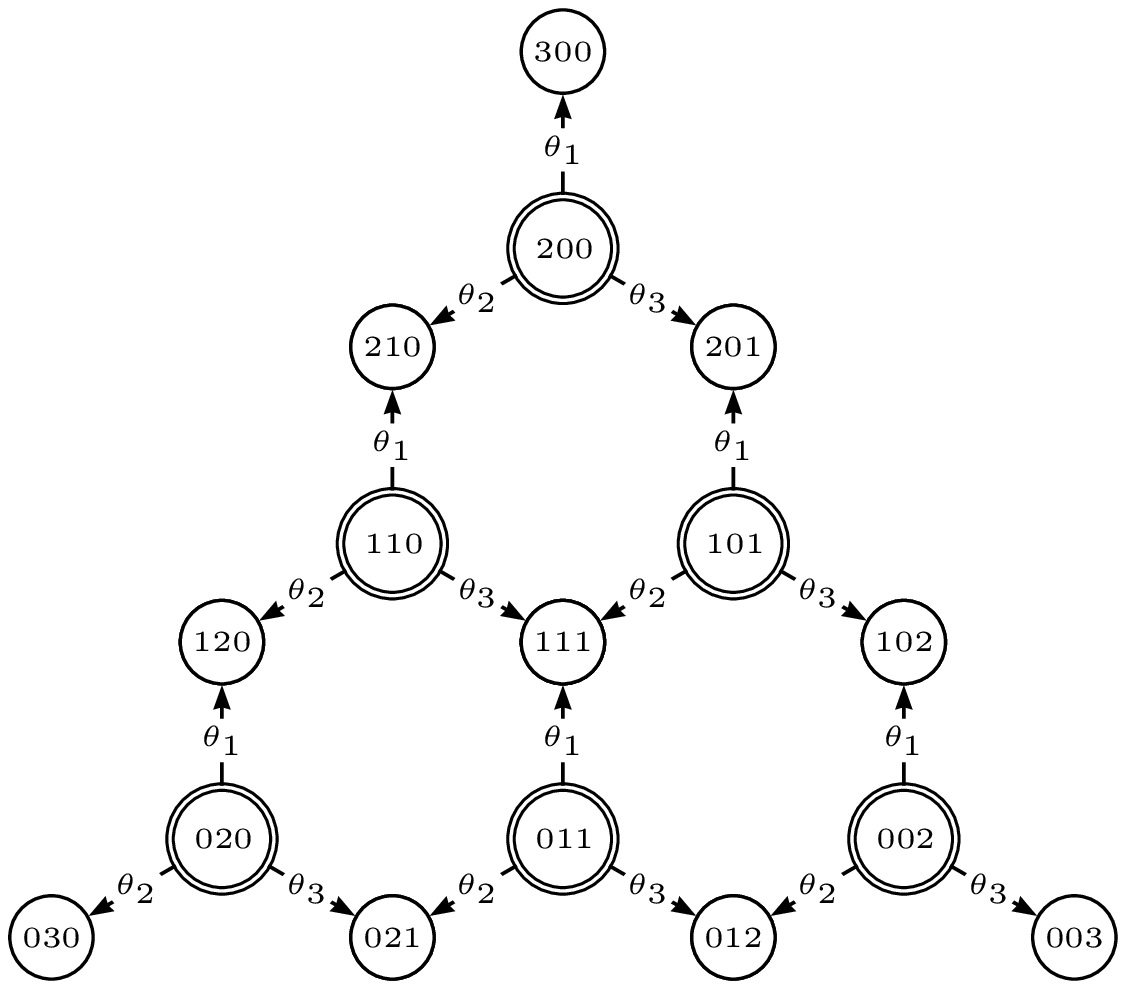}
        \caption{From $n=-2$ to $n=-3$}\label{fig:c}
     \end{subfigure}          
     \qquad 
     \begin{subfigure}{0.43\textwidth}  
        \centering
        \includegraphics[scale=0.53]{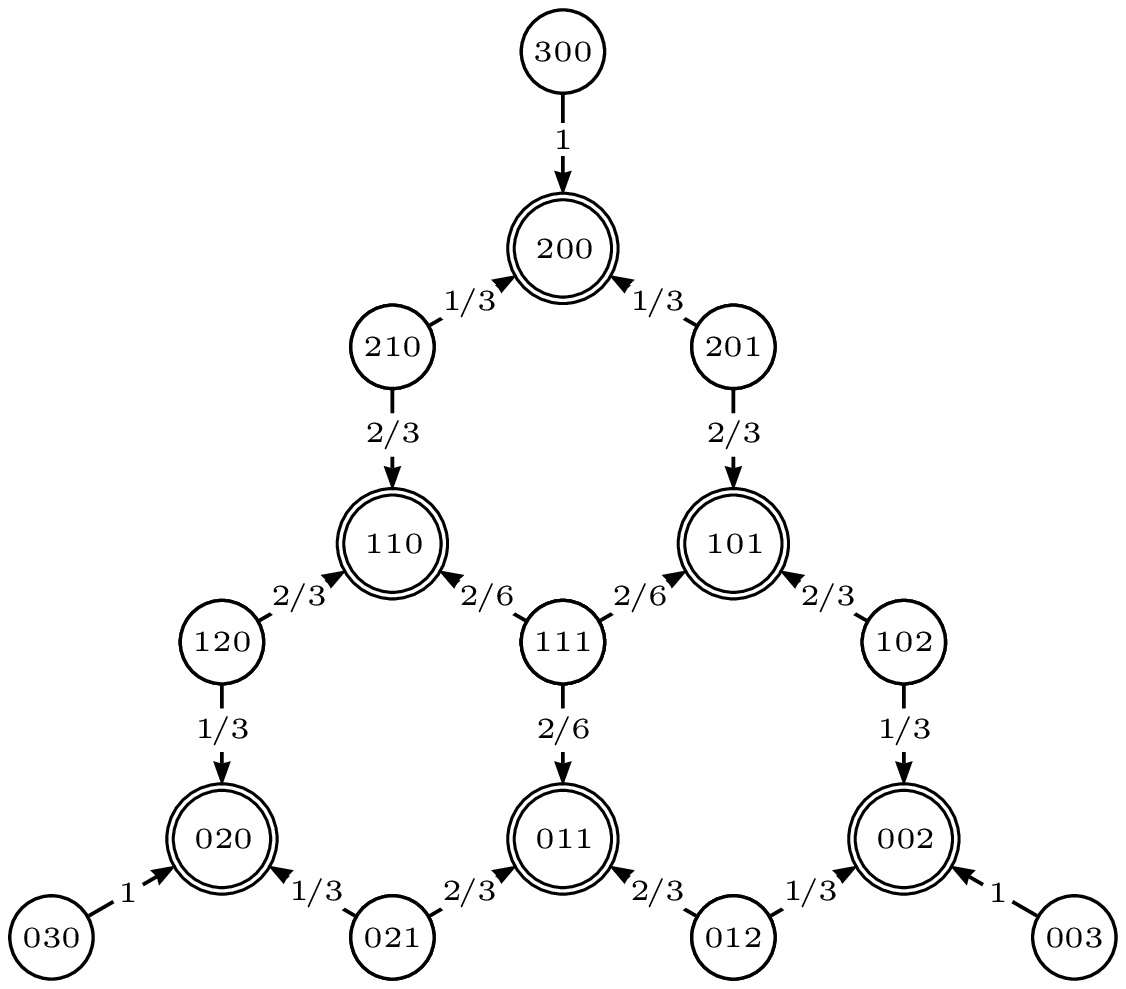}
        \caption{From $n=-3$ to $n=-2$}\label{fig:cc}
     \end{subfigure}    
 
 \caption{Random walk on the Pascal pyramid ($d=3$)} 
 \label{fig:MultiPascal}
 \end{figure}

From now on, we denote by ${(V_n^{d,\theta})}_{n \leq 0}$ 
the Markovian random walk corresponding to $(\theta_1, \ldots, \theta_d)$.
We write $V^{d,\theta}_n=\bigl(V^{d,\theta}_n(1), \ldots, V^{d,\theta}_n(d)\bigr)$. 
Each process ${\bigl(V^{d,\theta}_n(i)\bigr)}_{n \leq 0}$ is the random walk 
on the vertices of the Pascal graph as in Section~\ref{sec:Pascal}, and 
is Markovian with respect to $\GG$, that is, the filtration 
$\GG(i)$ generated by the process ${\bigl(V^{d,\theta}_n(i)\bigr)}_{n \leq 0}$ 
is immersed in $\GG$ 
(thus the multidimensional process satisfies conditions a) and b) of 
Definition~\ref{def:fullymonotonic}). 

It is worth mentioning that standardness of $\GG$ cannot be deduced from the 
equality $\GG=\GG(1) \vee \cdots \vee \GG(d)$ and from the fact that 
the filtrations $\GG(i)$ are standard and jointly immersed: This 
is a consequence of theorem~3.9 in~\cite{LauXLIII}, but 
Example~\ref{example:square} also provides a counter-example, 
and more precisely it shows that even the degeneracy of $\GG_{-\infty}$ cannot be 
deduced from the degeneracy of each $\GG_{-\infty}(i)$.

\begin{ppsition}
\label{prop:multi-standard}
The $d$-dimensional Pascal filtration generated by the process 
${\bigl(V^{d,\theta}_n\bigr)}_{n \leq 0}$ is standard for any $d \geq 2$ 
under any multinomial central measure. Consequently, the multinomial 
central measures are ergodic. 
\end{ppsition}

We now provide our three different proofs of the above proposition. 
Another proof is provided in \cite{LauEW}, by immersing the filtration 
in a filtration shown to be standard.

\subsection{First proof of standardness, using monotonicity of multidimensional Markov processes} 

%
%

Our first proof is an application of Theorem~\ref{thm:monotonicmulti}. 
Since we know that the tail sigma-algebra ${\cal G}_{-\infty}$ is trivial, 
it remains to show the monotonicity of the Markov process ${\left(V_n^{d,\theta}\right)}_{n\le0}$. 
Let us consider two points $v$ and $v'$ in $\Vb_n^d$: They satisfy
\[
  v(1)+\cdots+v(d)= v'(1)+\cdots+v'(d)=|n|.
\]
We want to construct a coupling of $\LL\left(V_{n+1}^{d,\theta}\given V_n^{d,\theta}=v\right)$ and $\LL\left(V_{n+1}^{d,\theta}\given V_n^{d,\theta}=v'\right)$ which is well-ordered with respect to $(v,v')$. We will get this coupling in the form $(Y,Y')=f_{n+1}(v,v',U)$, where $U$ is a uniform random variable on $\{1,\ldots,|n|\}$.
We can easily construct two partitions ${(A_i)}_{1\le i\le d}$ and ${(A'_i)}_{1\le i\le d}$ of $\{1,\ldots,|n|\}$ such that, for each $1\le i\le d$, $|A_i|=v(i)$, $|A'_i|=v'(i)$, and $A_i=A'_i$ whenever $v(i)=v'(i)$.
Now, for each $u\in\{1,\ldots,|n|\}$, there exists a unique pair $(i,i')$ such that $u\in A_i\cap A'_{i'}$ and we set   $f_{n+1}(v,v',u)\egdef (v-e_i,v'-e_{i'})$.
In this way, we respect the conditional distribution given in~\eqref{eq:multi-dim-Pascal}. Moreover, by construction it is clear that this coupling is well-ordered, since $v(i)=v'(i)$ implies $Y(i)=Y'(i)$, and at each step, coordinates never decrease by more than one unit.
Thus, since we know that ${\cal G}_{-\infty}$ is trivial, 
Theorem \ref{thm:monotonicmulti} applies and show that $\cal G$ is standard.

\subsection{Second proof of standardness, computing intrinsic metrics} 

In the preceding proof, we admitted the degeneracy of ${\cal G}_{-\infty}$. 
Here we provide an alternative short proof of standardness of 
the filtration which does not use this result. 
We have seen in the preceding proof that the Markov process is monotonic. 
It is even strongly monotonic (Definition~\ref{def:fullymonotonic}), thus  
we can use the tools of Section~\ref{sec:multimonotonic_kantorovich}. 
Moreover the Markov process is identifiable 
(see Section~\ref{sec:intrinsic}), hence 
Theorem~\ref{thm:fullymonotonicmulti} applies and then 
in order to derive standardness it suffices to check that 
$\LL(V_{-1} \given \FF_n) \to \LL(V_{-1})$, which is a straightforward 
consequence of the law of large numbers.  

We can use Proposition~\ref{ppsition:kantomulti} to derive the 
intrinsic metrics $\rho_n$, starting  at level $-1$. 
Remembering the unidimensional case, we get 
$$ \rho_n(v'_n, v''_n) = 
\sum_{i=1}^d a_i\frac{\bigl|v'_n(i)-v''_n(i)\bigr|}{|n|}.$$
The $V'$ property of $X_{-1}$ (Definition~\ref{lemma:VershikMarkov}), 
$$ \lim\limits_{n \to -\infty} \EE\bigl[\rho_n(V'_n, V''_n)\bigr] = 0,$$
is, similarly to $\LL(V_{-1} \given \FF_n) \to \LL(V_{-1})$, 
a straightforward consequence of the law of large numbers. 

\subsection{Third proof of standardness, constructing a generating parameterization}\label{sec:mpascal_standard1}

The third proof is a little bit longer, but it is self-contained (it does not use the degeneracy of  ${\cal G}_{-\infty}$, nor Theorem~\ref{thm:fullymonotonicmulti}). Moreover, it provides a generating parameterization of the 
$d$-dimensional Pascal filtration. 

We start by giving a natural parameterized representation of the Markov process ${(V_n^{d,\theta})}_{n \leq 0}$ and we will see that it is generating. 
We first introduce the notation 
$$\bar{v}(i)= \sum_{k=1}^i v(k)$$
for each $v \in \Vb_n^d$, any $n \leq 0$ and $i \in \{1, \ldots, d\}$. 
Recalling the Markovian transition from $n$ to $n+1$, we can easily construct a parameterized representation $(f_n,U_n)_{n\le 0}$  for the Markov process ${(V_n^{d,\theta})}_{n \leq 0}$ by taking the uniform distribution on 
$\{1, \ldots, |n|\}$ as the law of  $U_{n+1}$ and by defining the updating functions by 
$$ f_{n+1}\Bigl(v,u\Bigr) := v - e_i, $$
where $i$ is the unique index such that 
$u \in \bigl]\bar{v}(i-1), \bar{v}(i)\bigr]$

Now, we point out that, for each $1\le i\le d-1$, the process $\left(\bar{V}_n^{d,\theta}(i)\right)_{n\le0}$ is a Markov process with the same distribution as the process arising in the two-dimensional Pascal graph, that is, with our notations, 
$$
{\LL\left(\bar{V}_n^{d,\theta}(i)\right)}_{n\le0} 
= \LL{\left(V_n^{2,(p_i,1-p_i)}(i)\right)}_{n\le0}
$$
where $p_i:=\theta_1+\cdots+\theta_i$.
Moreover, the above parameterized representation of the Markov process ${(V_n^{d,\theta})}_{n \leq 0}$ provides a parameterized representation of the Markov process $\left(\bar{V}_n^{d,\theta}(i)\right)_{n\le0}$:
$$ \bar{V}_{n+1}^{d,\theta}(i) = f_{n+1}^{(i)}\Bigl(\bar{V}_n^{d,\theta}(i),U_{n+1}\Bigr) := 
\begin{cases}
  \bar{V}_{n}^{d,\theta}(i) -1 & \text{ if }U_{n+1} \leq \bar{V}_n^{d,\theta}(i) \\
  \bar{V}_{n}^{d,\theta}(i) & \text{ otherwise.}
\end{cases}
$$
This parameterization coincides with the increasing representation of the process that we used in the classical Pascal graph corresponding to $d=2$, hence as we have shown in Section~\ref{sec:Pascal}, Theorem~\ref{thm:monotonic} proves that it is a generating parameterization. It follows that for each $1\le i\le d$ and each $n\le0$, $\bar{V}_{n}^{d,\theta}(i)$ is measurable with respect to the $\sigma$-algebra generated by $U_n,U_{n-1},\ldots$. Thus 
$$ V_n^{d,\theta}(i) = \bar{V}_{n}^{d,\theta}(i)-\bar{V}_{n}^{d,\theta}(i-1) $$
is itself measurable with respect to the same $\sigma$-algebra, and the parameterized representation of the Markov process ${(V_n^{d,\theta})}_{n \leq 0}$ is generating. Lemma~\ref{lemma:generatingparam} then allows us to conclude that the $d$-dimensional Pascal filtration is standard.

\end{document}